\documentclass{amsart}
\usepackage[mathscr]{eucal}
\usepackage{graphicx}
\usepackage{amscd}
\usepackage{amsmath}
\usepackage{amsthm}
\usepackage{amsxtra}
\usepackage{calc} 
\usepackage{amsfonts}
\usepackage{amssymb}
\usepackage{latexsym}
\usepackage{pdfsync}
\usepackage[all]{xy}

\usepackage[colorlinks, bookmarks=true]{hyperref}

\newtheorem{theorem}{Theorem}[section]
\theoremstyle{plain}

\newtheorem{corollary}[theorem]{Corollary}
\newtheorem{corollary-definition}[theorem]{Corollary-Definition}

\newtheorem{definition}[theorem]{Definition}

\newtheorem{lemma}[theorem]{Lemma}

\newtheorem{proposition}[theorem]{Proposition}

\numberwithin{equation}{section}

\newcommand{\blank}{\hspace{0.04cm} \rule{2.4mm}{.4pt} \hspace{0.04cm} }

\DeclareMathOperator{\ot}{\overset{\rightharpoondown}{\otimes}}
\DeclareMathOperator{\otleft}{\overset{\leftharpoondown}{\otimes}}

\theoremstyle{definition}
\newtheorem{example}[theorem]{Example}
\newtheorem{remark}[theorem]{Remark}

\newcommand{\Ker}{\mathrm{Ker}\,}
\newcommand{\Coker}{\mathrm{Coker}\,}
\newcommand{\coker}{\mathrm{coker}\,}

\newcommand{\lrt}{\longrightarrow}
\newcommand{\Hom}{\mathrm{Hom}\,}
\DeclareMathOperator{\Ext}{\mathrm{Ext}}
\newcommand{\Tor}{\mathrm{Tor}}
\newcommand{\Imr}{\mathrm{Im}\,}
\newcommand{\Z}{\mathbb{Z}\,}
\newcommand{\Q}{\mathbb{Q}\,}

\newcommand{\Tr}{\mathrm{Tr}\,}

\newcommand{\lra}{\longrightarrow}

\newcommand{\Mod}{\mathrm{Mod}\,}

\newcommand{\ab}{\mathrm{Ab}}

\newcommand{\dia}[1]{\[\xymatrix{#1 }\]}
\usepackage{pgfpages}
\usepackage{tikz}
\usetikzlibrary{matrix,arrows,decorations,backgrounds}

\newcounter{hours}
\newcounter{minutes}

\begin{document}
%\blendcolors{!60!yellow}
%
%\pagecolor{myyellow}
%
%\pagecolor{red!1!green!1!blue!0.76!}

\title[]{Injective stabilization of additive functors. I. Preliminaries.}

\author{Alex Martsinkovsky}
\address{Mathematics Department\\
Northeastern University\\
Boston, MA 02115, USA}
\email{alexmart@neu.edu}
%\urladdr{}
\author{Jeremy Russell}
\address{Department of Mathematics\\
Rowan University\\
201 Mullica Hill Rd, Glassboro, NJ 08028}
\email{russelljj@rowan.edu}
%\urladdr{}
%\thanks{}
\date{\today, \setcounter{hours}{\time/60} \setcounter{minutes}{\time-\value
{hours}*60} \thehours\,h\ \theminutes\,min}
%\dedicatory{}

\begin{abstract}
This paper is the first one in a series of three dealing with the concept of injective stabilization of the tensor product and its applications. Its primary goal is to collect known facts and establish a basic operational calculus that will be used in the subsequent parts. This is done in greater generality than is necessary for the stated goal. Several results of independent interest are also established. They include, among other things, connections with satellites, an explicit construction of the stabilization of a finitely presented functor, various exactness properties of the injectively stable functors, a construction, from a functor and a short exact sequence, of a doubly-infinite exact sequence by splicing the injective stabilization of the functor and its derived functors. When specialized to the tensor product with a finitely presented module, the injective stabilization with coefficients in the ring is isomorphic to the 1-torsion functor. The Auslander-Reiten formula is extended to a more general formula, which holds for arbitrary (i.e., not necessarily finite) modules over arbitrary associative rings with identity. Weakening of the assumptions in the theorems of Eilenberg and Watts leads to characterizations of the requisite zeroth derived functors. 

The subsequent papers, provide applications of the developed techniques. Part~II deals with new notions of torsion module and cotorsion module of a module. This is done for arbitrary modules over arbitrary rings. Part~III introduces a new concept, called the asymptotic stabilization of the tensor product. The result is closely related to different variants of stable homology (these are generalizations of Tate homology to arbitrary rings). A comparison transformation from Vogel homology to the asymptotic stabilization of the tensor product is constructed and shown to be epic. 
 
\end{abstract}

\maketitle
\tableofcontents

\section{Introduction}

This is the first in a series of three papers dealing with the notion of injective stabilization of an additive functor. Of primary interest to us are the univariate tensor products, but eventually we will have to branch out to include other functors as well. Originally, this paper was meant to be a short section with a list of basic preliminaries in what is now the third paper. However, at some point the authors realized that the short section is no longer short and the idea to split it off into a separate paper started to emerge. This decision was eventually reinforced by two additional arguments. 
%The first one is related to the history of the subject, while the second one appeals to possible future developments. 

First, we remind the reader that the definition of the injective (and of the projective) stabilization of an additive functor was introduced by 
Auslander and Bridger in~\cite{AB}. Shortly before that, Auslander had developed the language of coherent functors in his fundamental work~\cite{A66}.
One can arguably claim that most of what one needs to know to profit from categories of coherent functors is already contained in those two sources. Yet, at the same time, the results proved or mentioned there are not always easy to extract when needed for practical purposes. A decision then has been made, with a general reader in mind, to write up the preliminaries, with special attention to detail, in a separate paper while keeping it as self-contained as reasonably possible. As a result, this paper is aimed at a reader who is familiar with the notion of additive functor, but might not have worked with functor categories. Several results from~\cite{A66} and~\cite{AB} have been included and streamlined, but a number of new results and examples have also been added. 

The other argument in favor of a separate paper is related to the direction of future research. For a long time, the first author has been calling for a study of stable categories. Most often this term refers to the category of modules modulo projectives. Its objects are modules, but the morphisms are quotients of the usual homomorphisms by the subgroup of homomorphisms factoring through projectives. This tool has numerous uses in diverse areas of representation theory, group cohomology, and topology (in fact, this concept originated the work of Eckmann and Hilton on duality in homotopy theory). But stable categories don't seem to have been studied for their own sake. An attempt at a phenomenological study of categories modulo projectives was recently undertaken in~\cite{MZ}. It then became clear that there were surprisingly tight and unexpected connections between the properties of the ring and the properties of its projectively stable category. Several years prior to that, the junior author of the present paper -- a graduate student at the time -- had made a simple but incisive comment that Hom modulo projectives is actually the projective stabilization of the covariant Hom functor. Thus the reader with a flexible attitude may say that properties of the ring are reflected, often in unexpected ways, in the properties of the projective stabilization of the covariant Hom functor. What has transpired during the work on this series of papers, is that the same can be said about the injective stabilization of the tensor product (and the projective stabilization of the contravariant Hom functor). Two unexpected applications of this philosophy -- the most general definitions of torsion and cotorsion modules of a module over an arbitrary ring -- will be given in the second paper of the series. The senior author is happy to admit that he was too timid in his call for study of stable categories. The new dictum should read ``study stable categories and stable functors''. It is to be hoped that this paper will be of help to those readers interested in functor categories who want to quickly start experimenting on their own.

We now give a brief outline of the contents of this paper. Motivation for the study of the injectively stabilized tensor product is provided in Section~\ref{motivation}.

Section~\ref{S:zeroth-derived} deals with what could be called ``homological algebra in degree 0''. It is the largest section of the paper and, as the name suggests, it deals with zeroth derived functors. Most of the material there is known in one form or another but is not easy to find in the literature. The formalism of the zeroth derived functors leads to one-line proofs of the theorems of Eilenberg and Watts.

Section~\ref{def-prop} contains the definition and basic properties of the injective stabilization of an arbitrary additive functor, which leads to injectively stable functors. In the terminology of Grothendieck, these are precisely effaceable functors. They form the torsion class of a hereditary torsion theory on the functor category. The torsion-free class consists of the mono-preserving functors. 

In Section~\ref{satellites}, we see a natural example of injectively stable 
functors, the right satellites. This is an important topic in its own right, primarily because the right and left satellites form an adjoint pair. They will reemerge in full strength in the third paper, but for now we record an important result: the injective stabilization of a half-exact functor is nothing but the counit of that adjunction. The new notion of cosatellite is also introduced there. 

The injective stabilization of a finitely presented (aka coherent) functor is investigated in Section~\ref{inj-fp}. The defect of such a functor appears there, which leads to a 4-term exact sequence of fundamental importance. The injectively stable finitely presented functors are precisely those with defect zero. 

Various exactness properties of the injective stabilization of a functor are described in Section~\ref{exact}. The injective stabilization of the tensor product and the harpoon notation for it make their first appearance there. 

Section~\ref{right-exact} specializes to right-exact functors. In that case, the injective stabilization of the functor admits yet another description: 
this is the first right satellite of the first left-derived functor of that functor. Given a short exact sequence of modules, the values of injective stabilization of a right-exact functor on the modules and their cosyzygy modules can be spliced together with the values of the left-derived functors on the modules. The resulting sequence is doubly-infinite and exact. 

The above results set up the stage for Section~\ref{inj-tensor-prod}, where we look at the injective stabilization of the tensor product in more detail. 
This is actually a bifunctor. As is the case with the tensor product itself, its  
injective  stabilization has an inert variable and an active variable (but no balance!). As was shown by Auslander, when the inert variable is finitely presented, the injective stabilization of the tensor product with the inert variable is isomorphic to the covariant functor $\Ext^{1}$ of the transpose of that inert variable. When evaluated on the ring, this yields the 1-torsion (=  Bass torsion) submodule of the inert variable. Finally, we establish a duality formula relating the injective stabilization of the tensor product and the projective stabilization of the contravariant Hom functor.  It is similar to the classical Auslander-Reiten formula, and in fact implies it. Unlike the Auslander-Reiten formula, it holds for arbitrary (i.e., not necessarily finitely presented) modules.

Section~\ref{S:small-functor-cat} exploits a remarkable property of additive functors defined on finitely presented modules -- they all preserve filtered colimits -- which allows to build an equivalence between the category of all functors on finitely presented modules and the category of filtered-colimit-preserving functors on the entire module category. This construction, called colimit extension, offers an alternative view on the functors on finitely presented modules.

\section{Motivation}\label{motivation}

The starting point for this series of papers was the desire to find a \texttt{homological} counterpart to Buchweitz's generalization of Tate \texttt{cohomology} to arbitrary rings~\cite{Bu}. A solution to that problem will be presented in the third paper of the series. The construction of Buchweitz is easy to describe:
\[
\mathrm{B}^{n}(M,N) := \varinjlim_{i}\, \underline{\Hom}(\Omega^{i+n}M, \Omega^{i}N),
\]
where $\underline{\Hom}$ stands for $\Hom$ modulo projectives, and $\Omega$ indicates the syzygy operation (which is an endofunctor on the category of modules modulo projectives, which makes the right-hand side well-defined).
To motivate our construction for a \texttt{homological} analog of $\mathrm{B}^{n}(M,N)$, we re-examine Buchweitz's definition from a different point of view. First, recall the notion of the \texttt{projective stabilization} of an additive functor. Let $F$ be an additive covariant functor from the category of modules over a ring to the category of abelian groups. Given a module $M$, let 
\[
P_1 \overset{\partial}{\longrightarrow} P_0 \overset{\pi}{\longrightarrow} M \longrightarrow 0
\]
be a projective presentation of~$M$. Applying $F$, one has $L_{0}F := \Coker F(\partial)$. The cokernel of the canonical natural transformation $L_{0}F \longrightarrow F$ is called the \texttt{projective stabilization} of $F$ and is denoted by 
$\underline{F}$. It is unique up to a canonical isomorphism and 
$\underline{F}(M) \cong \Coker F(\pi)$. In particular, if $F := (A, \blank)$ is the covariant Hom functor determined by the module $A$, then
\[
\underline{F}(M) \cong \Coker \big((A,P_0) \longrightarrow (A,M)\big)
\]
is precisely the component of the covariant Hom functor modulo projectives at $M$.

This point of view on Hom modulo projectives hints at a possible approach to constructing a homological counterpart of Buchweitz's version of stable cohomology: instead of computing the colimit of the projective stabilizations of the covariant Hom functors, one should compute the limit of the injective stabilizations of relevant tensor products. As in the cohomological setting, the choice of the resolutions is important: one of the variables will contribute a projective resolution, whereas the other -- an injective one.

\section{The zeroth derived functors: examples and applications}\label{S:zeroth-derived}

We begin by reviewing the basic definitions and properties of derived functors of additive functors. In fact, we focus our attention on the zeroth derived functors, a subject which is -- quite unfortunately -- not often discussed in the literature. As an application, we give one-line proofs of two theorems of Eilenberg and Watts. Let $F : \Mod(\Lambda) \to \ab$ be a covariant additive functor from the category of left $\Lambda$-modules to the category of abelian groups. Given a module $M$, let 
\[
\ldots \lra P_1 \overset{\partial}{\longrightarrow} P_0 \overset{p}{\longrightarrow} M \longrightarrow 0
\]
be a projective resolution of~$M$. Applying $F$, we have a complex 
\[
\xymatrix
	{
	\ldots \ar[r]
	& F(P_{2}) \ar[r]
	& F(P_{1}) \ar[r]
	& F(P_{0})
		}
\] 
whose homology in degree $i$ will be denoted by $L_{i}F(M)$. The following well-known results are elementary.

\begin{lemma}\label{L:Li}
For any additive functor $F$ and $i \in \Z$, 

 \begin{enumerate}
 \item The isomorphism class of $L_{i}F(M)$ is independent of the choice of the projective resolution of $M$.
 \smallskip
 
 \item Any homomorphism $f : M \to M^{\prime}$ of $\Lambda$-modules induces a homomorphism $L_{i}F(f) : L_{i}F(M) \to L_{i}F(M^{\prime})$ of abelian groups.
 \smallskip 
 
 \item $L_{i}F$ is an additive covariant functor $\Mod(\Lambda) \to \ab$, called the $i$th left-derived functor of $F$.
 \smallskip
 
 \item $L_{i}F$ vanishes on projectives for any $i \geq 1$. \qed
 \end{enumerate} 
\end{lemma}

Now we restrict our attention to the case $i = 0$. By the universal property of cokernels, we have a commutative diagram 
\[
\xymatrix
	{
	F(P_{1}) \ar[r]^{F(\partial)}
	& F(P_{0}) \ar[r] \ar[rd]^{F(p)}
	& L_{0}F(M) \ar[r] \ar[d]^{\lambda_{M}}
	& 0
\\
	&
	& F(M)
	}
\] 
in which the horizontal row is exact. The following results are well-known and easy to prove.

\begin{lemma}\label{L:L0}
For any additive functor $F$,
\begin{enumerate}

\item $L_{0}F$ is right-exact.
 \smallskip
 
 \item $\lambda : L_{0}F \lra F$ is a natural transformation.
 \smallskip
 
 \item $F$ is right-exact if and only if $\lambda$ is an isomorphism.
 \smallskip
 
  \item If $M$ is projective, then $\lambda_{M}$ is an isomorphism; in fact, 
 $\lambda_{M}$ can be chosen to be the identity. 
 \smallskip
 
 \item The natural transformation $L_{i} \lambda : L_{i}(L_{0}F) \lra L_{i}F$ is an isomorphism for all $i$.\footnote{Here $L_{i}\lambda$ denotes the natural transformation induced by $\lambda : L_{0}F \lra F$. Details are left to the reader.} In particular, $L_{0}\lambda : L_{0}(L_{0}F) \to L_{0}F$ is an isomorphism.
 \smallskip
  
 \item $\lambda : L_{0}F \lra F$ is universal with respect to natural transformations from right-exact functors to $F$, i.e., if $\alpha : G \lra F$ is a natural transformation with $G$ right-exact, then there is a unique 
 $\beta : G \lra L_{0}F$ making the diagram
\[
\xymatrix
	{
	& G \ar[d]^{\alpha} \ar@{.>}[ld]_{\exists !\, \beta} 
\\
	L_{0}F \ar[r]^{\lambda}
	& F
	}
\] 
commute. In other words, $\lambda$ induces an isomorphism 
$(G, L_{0}F) \overset{\simeq}\lra (G,F)$ of abelian groups.
\smallskip 
\item (Characterization of $L_{0}F$) If $\alpha : G \lra F$ is a natural transformation with $G$ right-exact and $\alpha$ evaluates to an isomorphism on projectives, then the unique transformation $\beta : G \lra L_{0}F$ from the diagram above is an isomorphism.  \qed
\end{enumerate}
\end{lemma}

Next we recall, for an additive covariant functor $F$, the construction of the natural transformation $\tau :F(\Lambda) \otimes \blank \lra F$. Here one uses the bimodule structure of $\Lambda$, which makes $F(\Lambda)$ a right $\Lambda$-module. Given a \texttt{left} $\Lambda$-module $B$ and $b \in B$, let $r_{b} : \Lambda \lra B$ be the map $l \mapsto lb$. Now define $\tau_{B} : F(\Lambda) \otimes B \lra F(B)$ by setting $\tau_{B} (x \otimes b) := F(r_{b})(x)$, where $x \in F(\Lambda)$.\footnote{Under the adjunction between the tensor product and the covariant Hom, $\tau_{b}$ corresponds to the map $B \to (F(\Lambda), F(B)) : b \mapsto F(r_{b})$.} By definition, when $B = \Lambda$, the term $F(r_{b})(x)$ is $xb$, the result of the  right action of $b \in \Lambda$ on $x \in F(\Lambda)$. Whence

\begin{lemma}\label{L:coun}
 $\tau_{\Lambda} : F(\Lambda) \otimes \Lambda \lra F(\Lambda)$ is the canonical isomorphism. \qed
\end{lemma}

The next result is an easy consequence of Lemma~\ref{L:coun}. 

\begin{lemma}[\cite{A66}, p. 227]\label{L:tau}
 If a covariant functor $F : \Mod(\Lambda) \to \ab$ commutes with coproducts, then $\tau : F(\Lambda) \otimes \blank \lra F$ is isomorphic to $\lambda : L_{0}F \lra F$. In particular,  
 \[
 L_{0}F \simeq F(\Lambda) \otimes \blank .
 \] 
\end{lemma}

\begin{proof}
It follows from the assumption that $F$ is additive.  The same assumption and Lemma~\ref{L:coun} imply that $\tau$ is an isomorphism on free modules and hence on projectives. The desired result now follows from Lemma~\ref{L:L0}, (7).
\end{proof}

As an immediate consequence of the preceding lemma, we have a theorem characterizing the tensor product. 

\begin{theorem}[Eilenberg, Watts]
 If a covariant functor $F$ commutes with coproducts and is right-exact, then
 \[
 \tau : F(\Lambda) \otimes \blank \lra F
 \]
 is an isomorphism.
\end{theorem}
\begin{proof}
 Immediately follows from Lemma~\ref{L:L0}, (3).
\end{proof}

\begin{example}
 Let $A$ be a $\Lambda$-module and $F : = (A, \blank)$. In this case we can describe $\tau: (A, \Lambda) \otimes \blank \lra (A, \blank)$ and its image.
 If $B$ is a  $\Lambda$-module, $f \in (A, \Lambda) =: A^{\ast}$, and $b \in B$, then $\tau_{B}$ is given by
 \[
 \tau_{B} :  A^{\ast} \otimes B \lra (A, B) : f \otimes b \mapsto r_{b} \circ f
 \]
 The image of $\tau_{B}$ consists of the maps $A \lra B$ factoring through finitely generated projectives. If either $A$ or $B$ is finitely generated, then the image of $\tau_{B}$ consists of \texttt{all} maps factoring through projectives, and therefore the cokernel of $\tau_{B}$ is isomorphic to $(\underline{A,B})$, the Hom modulo projectives. When $A$ is finitely presented, the latter is functorially isomorphic to $\Tor_{1}(\Tr A, B)$.
\end{example}

We want to examine the foregoing example in more detail. Recall the diagram
\[
\xymatrix
	{
	& A^{\ast} \otimes \blank \ar[d]^{\tau} \ar@{.>}[ld]_{\exists !\, \beta} 
\\
	L_{0}(A, \blank) \ar[r]^{\lambda}
	& (A, \blank)
	}
\] 
from Lemma~\ref{L:L0}, (6) with $F := (A, \blank)$ and $G := A^{\ast} \otimes \blank$. Straight from the definitions, one easily checks that for any module $B$, the image of $\lambda_{B}$ consists of all maps $A \lra B$ that factor through projectives. The image of $\tau_{B}$ consists of the maps $A \lra B$ that factor through finitely generated projectives. As we already saw, when $(A,\blank)$ commutes with coproducts, $\beta : A^{\ast} \otimes \blank \lra L_{0}(A, \blank)$ is an isomorphism. Hence, in that case, any map with domain $A$ which factors through a projective factors through a finitely generated projective. We shall now show that this condition characterizes the modules $A$ for which $\beta$ is an isomorphism.

\begin{theorem}\label{T:LHom}
 The natural transformation $\beta : A^{\ast} \otimes \blank \lra L_{0}(A, \blank)$ is an isomorphism if and only if any map with domain $A$ which factors through a projective factors through a finitely generated projective.
\end{theorem}

\begin{proof}
 The ``only if'' part has already been established. The prove the converse, it suffices, in view of Lemma~\ref{L:L0}, (7), to show that $\tau$ is an isomorphism on projectives or, equivalently, free modules. Let $N$ be a cardinal number and
 $P := \Lambda^{(N)}$. We need to show that $\tau_{P} : A^{\ast} \otimes P \lra 
 (A, P)$ is an isomorphism. Suppose $\tau_{P}(g) = 0$, where $g = \sum l_{i} \otimes h_{i}$. There are only finitely many components of $P$ containing the nonzero components of all of the $h_{i}$. Those components form a finitely generated direct summand $Q$ of $P$ and we have a commutative diagram 
 \[
\xymatrix
	{
	A^{\ast} \otimes Q \ar@{>->}[d] \ar[r]_{\tau_{Q}}^{\cong}
	& (A, Q) \ar@{>->}[d]
\\
	A^{\ast} \otimes P \ar[r]_{\tau_{P}}
	& (A,P)
	}
\] 
 where, by Lemma~\ref{L:coun}, $\tau_{Q}$ is an isomorphism. Also, the vertical maps are split monomorphisms. The element $\sum l_{i} \otimes h_{i}$ in 
 $A^{\ast} \otimes Q$ is mapped by the vertical arrow to $g$. Hence it is mapped to zero by the composition of $\tau_{Q}$ and the other vertical arrow, which shows that this element is zero. Thus $\tau_{P}$ is monic. To show that it is epic, pick a map $f \in (A,P)$. It trivially factors through the projective $P$ and, by the assumption, it factors through a finitely generated projective. Thus the image of $f$ is a submodule of a finitely generated direct summand $Q$ of $P$ and we get a commutative diagram as above. It follows that $f$ is in the image of $\tau_{P}$.
 \end{proof}
 
 The just proved theorem allows to describe the left-derived functors of a covariant Hom functor whose fixed argument has the property mentioned in the theorem.
 
\begin{corollary}
 Let $A$ be a $\Lambda$-module with the property that any map with 
 domain~$A$ which factors through a projective factors through a finitely generated projective (e.g., $A$ is finitely generated). Then 
\[
L_{i}(A, \blank) \simeq \Tor_{i}(A^{\ast}, \blank)
\]
for all $i$.
\end{corollary}

\begin{proof}
 Follows from Lemma~\ref{L:L0}, (5).
\end{proof}

\begin{remark}
 The zeroth left-derived functor of the covariant Hom functor with a \texttt{finitely presented} contravariant argument $A$ was determined in (\cite[Proposition 7.1]{A66}):
%\[
%L_{0}(A, \blank) \simeq A^{\ast} \otimes \blank
%\]
%(Here $A$ is a left $\Lambda$-module). 
%
%It now follows that
% \[
% L_{n} (C, \blank) \simeq \Tor_{n}(C^{\ast},\blank)
% \]
Thus the above corollary provides a more general result.
\end{remark}

Since the natural transformation $\beta : A^{\ast} \otimes \blank \lra L_{0}(A, \blank)$ is an isomorphism whenever $(A, \blank)$ preserves coproducts,  Theorem~\ref{T:LHom} imposes a necessary condition on $A$ for the functor  $(A, \blank)$ to preserve coproducts.

\begin{corollary}
 If $(A, \blank)$ preserves coproducts, then any map with domain $A$ which factors through a projective factors through a finitely generated projective. \qed
\end{corollary}

Similar arguments allow to describe the left-derived functors of the covariant 
$\Ext$ functor under a suitable finiteness condition on the fixed variable.

\begin{example}
 Let $A$ be a $\Lambda$-module and $F : = \Ext^{1}(A, \blank)$. The natural transformation $\tau: \Ext^{1}(A, \Lambda) \otimes \blank \lra \Ext^{1}(A, \blank)$ is easy to describe in the language of extensions. If $B$ is another $\Lambda$-module, $[x] \in \Ext^{1}(A, \Lambda)$ is the class represented by a short exact sequence $x$, and $b \in B$, then 
 \[
 \tau_{B} :  \Ext^{1}(A, \Lambda) \otimes B \lra \Ext^{1}(A, B)
 \]
 applied to $[x] \otimes b$ is just the class of the pushout of $x$ along $r_{b} : \Lambda \lra B$. Interpreting the values of $\Ext$ as homotopy classes of chain maps between shifted projective resolutions of the arguments, we have another description: $\tau_{B}([x] \otimes b) = [r_{b} \circ x]$. 
 
 Suppose now that $\Ext^{1}(A, \blank)$ commutes with coproducts. This happens, for example, if $A$ is of type $FP_{2}$, i.e., $A$ has a projective resolution whose terms in degrees from 0 to 2 are finitely generated.\footnote{See~\cite{Breaz-2013} for a general discussion of this question.} Arguing as before, we see that the natural transformation $\tau: \Ext^{1}(A, \Lambda) \otimes \blank \lra \Ext^{1}(A, \blank)$ is an isomorphism on projectives, which identifies $L_{0}\Ext^{1}(A, \blank)$ as 
 $\Ext^{1}(A, \Lambda) \otimes \blank$. As a consequence, under the foregoing assumption, for all $i$ we have
 \[
 L_{i}\Ext^{1}(A, \blank) \simeq \Tor_{i}(\Ext^{1}(A, \Lambda), \blank).
 \]
 \end{example}
 Similarly, if $A$ is of type $FP_{n+1}$, then for all $i$
 \[
 L_{i}\Ext^{n}(A, \blank) \simeq \Tor_{i}(\Ext^{n}(A, \Lambda), \blank).
 \]
 
Next we want to look at the \texttt{right-derived} functors of a covariant functor $F$. To describe $R^{0}F$, we assume that $F$ is finitely presented. Recall that functor is said to be \texttt{finitely presented} if it is isomorphic to the cokernel of a natural transformation between two representable functors (in particular, such a functor is necessarily additive). Thus a finitely presented covariant functor $F$ is determined by an exact sequence $(A, \blank) \lra (B, \blank) \lra F \lra 0$. By Yoneda's lemma, the transformation between the representable functors is of the form $f : B \to A$. The kernel of this map will be denoted by $w(F)$ and called the \texttt{defect} of $F$. Its isomorphism type is uniquely determined by $F$. To see this, we first observe that the category $fp(\Mod(\Lambda), \ab)$ of finitely presented covariant functors from $\Lambda$-modules to abelian groups is abelian~\cite[Proposition~3.2]{A82}, with the usual notions like kernel, cokernel, exactness, etc, defined componentwise. Yoneda's lemma shows that the representables are precisely the projectives in that category. In particular, the defining sequence for~$F$ is just a projective presentation. From basic homological algebra, we know that it is unique up to homotopy equivalence. By Yoneda's lemma, that equivalence passes to the maps $f : B \to A$. This can be viewed as a two-term complex, and~$w(F)$ is just a homology group of that complex. As such, it is homotopy-invariant. The same argument shows that $w: fp(\Mod(\Lambda), \ab) \lra \Mod(\Lambda)$ is a \texttt{contravariant} functor. The snake lemma shows that the defect $w$ is left-exact\footnote{The horseshoe lemma and the fact that the projective dimension of a finitely presented functor is at most 2 show that the defect $w(\blank)$ is actually an exact functor.} and, in particular, additive.  
 
 The universal property of cokernels yields a commutative diagram 
 \[
\xymatrix
	{
	(A, \blank) \ar[r] 
	& (B, \blank) \ar[r] \ar[rd]
	& F \ar[r] \ar[d]^{\rho}
	& 0
\\
	&
	& (w(F), \blank)
	}
\] 
where $\rho$ is clearly an isomorphism on injectives. Together with the fact that 
$(w(F), \blank)$ is left-exact, we have   

\begin{proposition}[\cite{A66}, top of page 210]\label{P:rho-defect}
 The unique comparison transformation  $R^{0}F \lra (w(F), \blank)$ is an isomorphism. \qed
\end{proposition}

\begin{corollary}
 The induced transformation $R^{i}F \lra \Ext^{i}(w(F), \blank)$ is an isomorphism for all $i$. \qed
\end{corollary}

\begin{example}
 Let $F := A \otimes \blank$, where $A$ is finitely presented with presentation 
 $Q \lra P \lra A \lra 0$. An easy calculation produces a finite presentation for $F$:
 \[
 (Q^{\ast}, \blank) \lra (P^{\ast}, \blank) \lra A \otimes \blank \lra 0
 \]
 which shows that $w(A \otimes \blank) \simeq A^{\ast}$, and therefore 
 \[
 R^{0}(A \otimes \blank) \simeq (A^{\ast}, \blank)
 \]
 Moreover, the transformation $\rho$ from the diagram above coincides with the canonical transformation $A \otimes \blank \lra (A^{\ast}, \blank)$.
\end{example}

\begin{corollary}
 If $A$ is finitely presented, then 
 \[
 R^{i}(A \otimes \blank) \simeq \Ext^{i}(A^{\ast}, \blank)
 \]
 for all~$i$. \qed
\end{corollary}

Now that the first examples of ``derived-on-the-wrong-side'' covariant functors have been examined, it is natural to look at contravariant functors. Before doing this, it is convenient to set the following notation. For covariant functors $F$, we adhere to the standard notation $L_{i}F$ when projective resolutions are used and $R^{i}F$ when injective resolutions are used. When dealing with derived functors of contravariant functors, the choice of the letter is flipped (this is still standard) but we also flip subscripts and superscripts. Thus, for a contravariant~$F$, the left-derived functors are denoted by $L^{i}F$ and the right-derived functors are denoted by $R_{i}F$. The table below, whose column captions indicate the choice of the resolution, encapsulates our conventions.

\begin{center}
 \begin{tabular}{|c|c|c|}
\hline  
	& Projective resolutions 
	& Injective resolutions 
\\
\hline 
	Covariant $F$ 
	& $L_{i}F$, \quad $L_{0}F \lra F \lra$  p-stab 
	& $R^{i}F$, \quad i-stab $ \lra F \lra R^{0}F$ 
\\
\hline 
	Contravariant $F$ 
	& $R_{i}F$, \quad  i-stab $\lra F \lra R_{0}F$ 
	& $L^{i}F, \quad L^{0}F \lra F \lra$ p-stab 
\\
\hline 
\end{tabular}
\end{center}
\bigskip
Thus, the subscripts are indicative of projective resolutions, whereas the superscripts are indicative of injective resolutions. The arrows are the natural transformations, and the abbreviations p-stab and i-stab stand for projective and, respectively, injective stabilization -- two concepts that will be the main objects of study in this series of papers. A potentially confusing aspect of the adapted nomenclature is that, in the contravariant case, the injective stabilization requires a projective resolution and the projective stabilization requires an injective resolution. For the reader who is sensing the presence of adjoints, we can offer  
a better mnemonic rule: projective stabilizations are defined as the cokernels of  the counits, whereas injective stabilizations are defined as the kernels of the units. Detailed explanations will be provided in the subsequent sections and papers.

With the notation set, in the remainder of this section, we switch our attention to the right-derived functors of a contravariant functor. Thus we are going to apply contravariant functors to \texttt{projective resolutions}, and the resulting homology will be decorated by the symbols $R_{i}$.

First, we briefly indicate how Lemmas~\ref{L:Li} and~\ref{L:L0} should be modified to fit our new choice. As we just said, the $L_{i}$ should be replaced by the $R_{i}$. In Lemma~\ref{L:Li},~(2), the second arrow should be reversed. In part (3) of the same lemma, the word ``covariant'' is replaced with ``contravariant'' and ``left-derived'' should now read ``right-derived''. Part (4) remains unchanged. The arrows in the diagram after Lemma~\ref{L:Li} should be reversed and the component $\lambda_{M}$ will now be called 
$\rho_{M}$.

Lemma~\ref{L:L0}, (1) should now read ``$R_{0}F$ is left-exact''. The arrow in (2) should be reversed. The new transformation is denoted by $\rho$. In (3), ``right-exact'' should read ``left-exact''. In (4), $\lambda_{M}$ becomes $\rho_{M}$. The arrows in (5) should be reversed. For the convenience of the reader, we are going to restate parts (6) and (7) in full detail.

\begin{itemize}
 \item[(6)] $\rho : F \lra R_{0}F$ is universal with respect to natural transformations from $F$ to left-exact functors, i.e., if $\alpha : F \lra G$ is a natural transformation with $G$ left-exact, then there is a unique 
 $\beta : R_{0}F \lra G$ making the diagram
 \[
\xymatrix
	{
	F \ar[d]_{\alpha} \ar[r]^{\rho} 
	& R_{0}F \ar@{..>}[ld]^{\exists \,! \beta}
\\
	G
	}
\] 
 commute. In other words, $\rho$ induces an isomorphism $(R_{0}F, G) 
 \overset{\simeq} \lra (F,G)$ of abelian groups.
 \smallskip
 
 \item[(7)] (Characterization of $R_{0}F$) If $\alpha : F \lra G$ is a natural transformation with $G$ left-exact and $\alpha$ evaluates to an isomorphism on projectives, then the unique transformation $\beta : R_{0} F \lra G$ from the diagram above is an isomorphism. 
 \end{itemize}

%the right-derived functors of covariant functors. Actually, this is the main subject of this paper and we shall start doing this in the next section. 

Let $F : \Mod(\Lambda) \to \ab$ be an additive contravariant functor from the category of left $\Lambda$-modules to the category of abelian groups. We recall the construction of the natural transformation $\sigma : F \lra (\blank , F(\Lambda))$. Here, once again, one uses the bimodule structure of $\Lambda$, which makes $F(\Lambda)$ a left $\Lambda$-module. Given a left $\Lambda$-module $M$ and $a \in M$, let $r_{a} : \Lambda \lra M$ be the map $l \mapsto la$. Now define $\sigma_{M} : F(M) \lra (M, F(\Lambda))$ by setting 
$[\sigma_{M}(x)](a) := F(r_{a})(x)$, where $x \in F(M)$.\footnote{Under the self-adjunction of the contravariant Hom functor, $\sigma_{M}$ corresponds to the map $M \to (F(M), F(\Lambda)) : a \mapsto F(r_{a})$.} By definition, when $M = \Lambda$, the term $F(r_{a})(x)$ is $ax$, the result of the left action of $a \in \Lambda$ on $x \in F(\Lambda)$. Whence

\begin{lemma}\label{L:un}
$\sigma_{\Lambda} : F(\Lambda) \lra (\Lambda, F(\Lambda))$ is the canonical isomorphism. \qed
\end{lemma}

As an easy consequence, we have
\begin{proposition}\label{P:R0}
 If a contravariant functor $F$ converts coproducts into products, then $\sigma : F \lra (\blank , F(\Lambda))$ is isomorphic to $\rho : F \lra R_{0}F$. In particular, 
\[
R_{0}F \simeq (\blank, F(\Lambda)).
\]
\end{proposition}

\begin{proof}
 It follows from the assumption that $F$ is additive. The same assumption and Lemma~\ref{L:un} imply that $\sigma$ is an isomorphism on free modules and hence on projectives. The desired result now follows from~(7) above. 
\end{proof}

As an immediate consequence of the preceding lemma, we have a theorem characterizing the contravariant Hom functor.

\begin{theorem}[Eilenberg, Watts]
 If a contravariant functor $F$ converts coproducts into products and is left-exact, then
 \[
 \sigma : F \lra (\blank, F(\Lambda))
 \]
 is an isomorphism.
\end{theorem}

\begin{proof}
 Immediately follows from the fact that $\rho$ is an isomorphism when $F$ is left-exact.
 \end{proof}

 Defect can also be defined for finitely presented contravariant functors; in that case we shall use the notation $v(F)$. The properties of $v$ are analogous to those of~$w$.  Thus if $(\blank, A) \overset{(\blank, f)}\lra (\blank, B) \lra F \lra 0$ is a finite presentation, then we have a defining exact sequence 
 \[
 A \overset{f}\lra B \lra v(F) \lra 0
 \]
 The defect $v(\blank)$ is a covariant functor from finitely presented contravariant functors to modules. The snake lemma shows that $v$ is right-exact\footnote{The horseshoe lemma and the fact that the projective dimension of a finitely presented functor is at most 2 show that the defect $v(\blank)$ is actually an exact functor.} and, in particular, additive.
 
 The universal property of cokernels yields a commutative diagram 
 \[
\xymatrix
	{
	(\blank, A) \ar[r] 
	& (\blank, B) \ar[r] \ar[rd]
	& F \ar[r] \ar[d]^{\rho}
	& 0
\\
	&
	& (\blank, v(F))
	}
\] 
The natural transformation $\rho : F \lra (\blank, v(F))$ is obviously an isomorphism on projectives. Whence

\begin{proposition}
Let $F$ be a finitely presented contravariant functor. The unique comparison transformation $R_{0}F \lra (\blank, v(F))$ is an isomorphism. \qed
\end{proposition}

%It is clear that a finitely presented contravariant functor converts coproducts into products. Whence

The above diagram implies

\begin{corollary}
The canonical map $\rho_{\Lambda}$  is an isomorphism 
$F(\Lambda) \simeq v(F)$. \qed
\end{corollary}

%\begin{proof}
% Follows from Proposition~\ref{P:R0} and Yoneda's lemma.
%\end{proof}

\begin{corollary}
The induced transformation $R_{i}F \lra \Ext^{i}(\blank, v(F))$ is an isomorphism for all $i$. \qed 
\end{corollary}

\section{Injective stabilization: definition and basic properties}\label{def-prop}

We begin by reviewing the notion of the injective stabilization of an additive covariant functor from the category of (say, left) modules over a ring 
$\Lambda$ to the category of abelian groups (or, more generally, between two abelian categories, with the domain category having enough injectives). For an arbitrary module $B$, choose an injective copresentation

\begin{equation}\label{inj-copres}
 0 \longrightarrow B \overset{\iota}\longrightarrow I^0 \overset{\partial}{\longrightarrow} I^1
\end{equation}
and apply an additive functor $F$. Then the image of $F(\iota)$ is in $ \Ker F(\partial) = R^0F(B)$, and one obtains a natural transformation $\rho_{F} : F \longrightarrow R^0F$.

%\begin{remark}
%It is well-known that $R^{0}F$ is left-exact. Therefore, if $\rho_{F}$ is an isomorphism, then $F$ is left-exact. Conversely, if $F$ is left-exact, then, as is also well-known, $\rho_{F}$ is an isomorphism. Thus 
%$\rho_{F} : F \longrightarrow R^{0}F$ is an isomorphism if and only if $F$ is left-exact. 
%\end{remark}

The kernel of $\rho_{F}$ is called the \texttt{injective stabilization} of $F$ and is denoted by $\overline{F}$. Thus we have an exact sequence of functors and natural transformations
\begin{equation}\label{rho}
 0 \longrightarrow \overline{F} \longrightarrow F \overset{\rho_{F}}\longrightarrow R^{0}F
\end{equation}
It is immediate that $\overline{F}$, being a subfunctor of an additive functor, is itself additive.

Let 
\begin{equation}\label{inj-copres-prime}
 0 \longrightarrow B \overset{\iota'}\longrightarrow I^{0'} \overset{\partial'}\longrightarrow I^{1'}
\end{equation}
be another injective copresentation of $B$. Extending the identity map on $B$ to injective resolutions extending the chosen presentations, we have a homotopy equivalence~$\varepsilon$ between the two. Since $F$ is additive, $F(\varepsilon)$ is also a homotopy equivalence, and therefore the induced map $R^{0}F(B) \longrightarrow R^{0'}F(B)$ on the homology is an isomorphism. We now have a commutative diagram 
 \[
\xymatrix
	{
	0 \ar[r]
	& \overline{F}(B) \ar[r] \ar@{.>}[d]
	& F(B) \ar[r] \ar@{=}[d]
	& R^{0}F(B) \ar[d]^{\cong}
\\
	0 \ar[r]
	& \overline{F}'(B) \ar[r]
	& F(B) \ar[r]
	& R^{0'}F(B)
	}
\] 
where the dotted map is an isomorphism. Since the horizontal arrows into $F(B)$ are the inclusions of the underlying sets, the dotted isomorphism is in fact an equality. In other words, we have uniqueness in the strongest possible sense. In summary:

\begin{lemma}\label{L:canonical}
 Any two choices for $\overline{F}(B)$ are equal. This equality is induced by any extension of the identity map on $B$ to the two injective resolutions of $B$. \qed
\end{lemma}

\begin{remark}\label{R:bimodule}
By definition, $\overline{F}(B)$ is an abelian group. If $\Gamma$ is another ring and $B$ is a $\Lambda$-$\Gamma$ bimodule, then, since $\overline{F}$ is additive, $\overline{F}(B)$ is also a right $\Gamma$-module. Moreover, the inclusion $\overline{F}(B) \to F(B)$ is a homomorphism of right $\Gamma$-modules, i.e., $\overline{F}(B)$ is a $\Gamma$-submodule of $F(B)$.\footnote{The reader should notice that here we used the fact that the right multiplication by an element of~$\Gamma$, being a homomorphism of left $\Lambda$-modules, extends to the injective resolution of $B$, and the induced map on $R^{0}F(B)$ is independent of the extension.}
\end{remark}

The commutative diagram 
\begin{equation}\label{defining-diagram}
\begin{gathered}
 \xymatrix
	{
	0 \ar[dr]
	&
	& 0 \ar[d]
	&
	&
\\
	&\Ker F(\iota) \ar[dr]
	& \overline{F}(B) \ar@{.>}[l]_>>>>>>>{\cong}  \ar[d]
	&
	&
\\ 
	&
	& F(B) \ar[d]_{\rho_{F}(B)} \ar[dr]^{F(\iota)}
	&
	&
\\	
	&0 \ar[r]
	& R^{0}F(B) \ar[r]
	& F(I^{0}) \ar[r]^{F(\partial)}
	& F(I^{1})
	}
\end{gathered}
\end{equation}
of solid arrows, whose rows, column, and diagonal are exact, shows that the dotted canonical map 
%$ \overline{F}(B) \longrightarrow \Ker F(\iota)$ 
is an isomorphism, thereby providing a practical way of computing $\overline{F}(B)$. In particular, one does not need an injective copresentation of $B$ -- it suffices to have an embedding~$\iota$ of~$B$ in an injective module. Whence

\begin{lemma}\label{L:altinjdef}
 In the above notation, $\overline{F}(B) = \Ker F(\iota)$. \qed
\end{lemma}

\begin{example}
 Let $\Bbbk$ be a field of characteristic zero and $\Lambda := \Bbbk[[T]]$. 
 Let $B : = \Bbbk[[T]]/(T)$; as a $\Lambda$-module, this is just~$\Bbbk$ with a trivial action of $T$.  Let $I : = \Bbbk[X]$, which we view as a $\Lambda$-module with 
 $T \cdot f : = df/dX$ for any $f \in I$. The inclusion $\iota : B \to I$ is clearly a homomorphism of $\Lambda$-modules. Moreover, this homomorphism is an essential extension -- this follows from the fact that, since $\Bbbk$ is of characteristic zero, the last nonzero higher-order derivative of a polynomial is a constant. Using again the fact that 
 $\mathrm{char}\, \Bbbk = 0$, it is not difficult to see that $I$ is a divisible 
 $\Lambda$-module, and is therefore injective. As a result, $\iota : \Bbbk \to \Bbbk[X]$ is the injective envelope. Let~$A$ be a finitely generated 
 $\Lambda$-module. We wish to compute $\Ker (A \otimes \iota)$.
 By the structure theorem for finitely generated modules over a PID, it suffices to assume that $A$ is cyclic. If $A$ is free, then $A \otimes \blank$ is an exact functor, and $\Ker (A \otimes \iota) = \{0\}$. Thus assume that 
 $A$ is not free, i.e., $A \simeq \Bbbk[[T]]/(T^{n})$ for some $n \geq 1$. Now we need to compute the kernel of 
 \[
\Bbbk[[T]]/(T^{n}) \otimes \Bbbk \lrt  \Bbbk[[T]]/(T^{n}) \otimes \Bbbk [X]
 \]
 Notice that $\Bbbk[[T]]/(T^{n}) \otimes \Bbbk [X] \simeq \Bbbk[X]/ (T^{n})\Bbbk[X]$, which is a zero module since $\Bbbk[X]$ is divisible and therefore 
 $(T^{n})\Bbbk[X] = \Bbbk[X]$. It follows that $\Ker (A \otimes \iota) \simeq \Bbbk$ for a cyclic~$A$. For a general $A$, $\Ker (A \otimes \iota) \simeq \Bbbk^{s}$, where $s$ is the number of nonzero invariant factors of~$A$.
\end{example}

The diagram~\eqref{defining-diagram} also yields

\begin{lemma}\label{L:vanish-inj}
 The injective stabilization of an additive functor vanishes on injectives. \qed
\end{lemma}

\begin{remark}
 The above lemma shows that there is no additive functor whose injective stabilization is isomorphic to a nonzero covariant Hom functor. Indeed, if $M$ is a nonzero module, then $(M, \blank)$ is nonzero on the injective envelope of  $M$, whereas any injective stabilization of a functor vanishes on it.
\end{remark}

The injective stabilization of an additive functor can also be defined by an extremal property.

\begin{proposition}\label{P:largest}
 The injective stabilization of an additive functor $F$ is the largest subfunctor of $F$ vanishing on injectives. 
\end{proposition}

\begin{proof} 
 In view of Lemma~\ref{L:vanish-inj}, we only need to show that any subfunctor $G$ of $F$ vanishing on injectives is a subfunctor of $\overline{F}$. Applying $F$ and $G$ to the injective copresentation \eqref{inj-copres}, we have a commutative diagram of solid arrows with exact rows
 \[
\xymatrix
	{
	0 \ar[r]
	& \overline{G}(B) \ar@{^{(}->}[r] \ar@{.>}[d]
	& G(B) \ar[r]^{G(\iota)} \ar@{^{(}->}[d]
	& G(I) \ar@{^{(}->}[d]
\\
	0 \ar[r]
	& \overline{F}(B) \ar@{^{(}->}[r]
	& F(B) \ar[r]^{F(\iota)}
	& F(I) 
	}
\] 
which induces a monic dotted arrow and thus shows that 
$\overline{G}$ is a subfunctor of $\overline{F}$. As $G(I) = \{0\}$, 
$G = \overline{G}$, and therefore $G$ is a subfunctor of $\overline{F}$.
\end{proof}

Now we want to discuss the question of ``size'' of $\overline{F}$. More precisely, we want to know when $\overline{F}$ is smallest or largest possible or, in other words, when $\overline{F} = 0$ or, respectively, $\overline{F} = F$. These possibilities are realized when $\rho_{F} : F \longrightarrow R^{0}F$ is monic or, respectively, the zero transformation.

\begin{lemma}\label{L:mono}
 $\overline{F} = 0$ if and only if $F$ preserves monomorphisms.
\end{lemma}

\begin{proof}
 The ``if'' part is immediate from the definition. For the ``only if'' part, suppose $i : B \longrightarrow C$ is monic. Then $\iota : B \longrightarrow I^{0}$ extends over $i$, and hence $F(\iota)$ extends over $F(i)$. By assumption, the former is monic, and therefore so is the latter. 
\end{proof}

\begin{corollary}\label{C:tensor-flat}
Let $L$ be a right $\Lambda$-module. The injective stabilization of the functor 
$L \otimes \blank$ is zero if and only if $L$ is flat. \qed
\end{corollary}

\begin{corollary}
 For any $\Lambda$-module $A$, the injective stabilization of the covariant 
 $\Hom$ functor $(A, \blank)$ is zero. \qed
\end{corollary}

\begin{corollary}
 For any additive functor $F$, the injective stabilization of $R^{0}F$ or of any of its subfunctors is zero. In particular, this holds for the image of the natural transformation $\rho_{F} : F \longrightarrow R^{0}F$.
\end{corollary}

\begin{proof}
 $R^{0}F$ is left-exact and, in particular, preserves monomorphisms. Any subfunctor of such a functor also preserves monomorphisms.
\end{proof}

Now we look at the other extreme, when the inclusion $\overline{F} \longrightarrow F$ is an isomorphism, i.e., $\rho_{F} : F \longrightarrow R^{0}F$ 
is the zero transformation.

\begin{definition}[\cite{AB}]
The functor $F$ is said to be \texttt{injectively stable}  if the natural transformation $\overline{F} \longrightarrow F$ is an isomorphism.
\end{definition}

The next result is straightforward.

\begin{proposition}[\cite{AB}, Proposition 1.7]\label{P:inj-stable}
 The following conditions are equivalent:

\begin{enumerate}
 \item[(a)] $F$ is injectively stable.
 \smallskip
 \item [(b)] $F(I) = 0$ for all injectives $I$.
 \smallskip
 \item[(c)]  $R^{0}F = 0$. \qed
\end{enumerate}
\end{proposition}

\begin{remark}
 Under our assumption that the domain category have enough injectives, Proposition~\ref{P:inj-stable} implies that injectively stable functors are precisely the \texttt{effaceable} functors of Grothendieck~\cite{Gr}.
\end{remark}

\begin{example}
For any $\Lambda$-module  $A$, the functor $\Ext^{1}(A, \blank)$ is injectively stable.
\end{example}

Since $\overline{F}(I) \simeq 0$ whenever $I$ is injective, we have

\begin{corollary}
 The injective stabilization of an additive functor is injectively stable. 
\end{corollary}
\begin{proof}
 Follows from Lemma~\ref{L:vanish-inj} and Proposition~\ref{P:inj-stable}.
\end{proof}

We can now rephrase Proposition~\ref{P:largest}.

\begin{corollary}
The injective stabilization of an additive functor is its largest injectively stable subfunctor. \qed
\end{corollary}

As another consequence of Proposition~\ref{P:inj-stable}, we have

\begin{lemma}
 Injectively stable functors form a Serre class, i.e., in a short exact sequence\footnote{To avoid set-theoretic difficulties, exactness, as well as all other requisite concepts, will be understood in a  componentwise sense.} 
 $0 \to F_{1} \to F \to F_{2} \to 0$ of additive functors, $F$ is injectively stable if and only if so are $F_{1}$ and $F_{2}$. 
\end{lemma}

\begin{proof}
This is an immediate consequence of the componentwise exactness of the sequence.
\end{proof}

The above observations lead to a torsion theory on the (possibly, large) category of additive functors. For an additive functor $F$ we have, in the notation of~\eqref{rho},
 an exact sequence
 \[
\xymatrix
	{
	0 \ar[r]
 	& \overline{F} \ar[r]
 	& F \ar[r]
 	& \Imr \rho_{F} \ar[r]
 	& 0
	}
\] 
Recapping the above discussion, we have that $ \overline{F}$ is injectively stable and $\Imr \rho_{F}$ preserves monomorphisms. Let $\mathscr{T}$ be the class of injectively stable functors and $\mathscr{F}$ the class of mono-preserving functors.

\begin{proposition}\label{P:torsion}
$(\mathscr{T}, \mathscr{F})$ is a hereditary torsion theory in the (possibly, large) category of additive functors.
\end{proposition}

\begin{proof}
Clearly, $\mathscr{T}$ is closed under subobjects. In view of the above short exact sequence, we only need to show that any natural transformation
$\alpha : T \to F$ from an injectively stable functor $T$ to a mono-preserving functor $F$ is zero. Notice that $\alpha(T)$, being a subfunctor of the mono-preserving functor $F$, is itself mono-preserving, and therefore, by Lemma~\ref{L:mono}, its injective stabilization is zero.  On the other hand,
Proposition~\ref{P:inj-stable} shows that, under any transformation, the image of an injectively stable functor is still injectively stable. Thus $\alpha(T)$, being equal to its own injective stabilization, must be zero.
\end{proof}

The foregoing argument leads to another useful characterization of injective stabilization.

\begin{proposition}\label{P:maps-to-mono-preserving}
 Let $F$ and $F'$ be additive functors and suppose that $F'$ is mono-preserving. If a natural transformation $ \gamma : F \to F'$  evaluates to an isomorphism on each injective, then 
 $\ker \gamma$ is the injective stabilization of~$F$. 
\end{proposition}

\begin{proof}
By Proposition~\ref{P:largest}, we need to show that $\ker \gamma$ is the largest subfunctor of $F$ vanishing on injectives. The vanishing property of 
$\ker \gamma$ follows immediately from the assumption. Now suppose that 
$i : G \hookrightarrow F$ is a subfunctor of $F$ vanishing on injectives. 
By Proposition~\ref{P:torsion}, the composition $\gamma i : G \lra F'$ is zero, and therefore $i$ lifts over $\ker \gamma$. Since $i$ is monic, so is the lifting, i.e., $G$ is a subfunctor of $\ker \gamma$.
\end{proof}

\begin{remark}
 The reader can rightfully expect that Proposition~\ref{P:torsion} and 
 Proposition~\ref{P:maps-to-mono-preserving} have their analogs for the projective stabilization $\underline{F}$ of an additive covariant functor~$F$, which is defined by the exact sequence 
 \[
 L_{0}F \overset{\lambda_{F}} \lrt F \lrt \underline{F} \lrt 0.
 \]
 This leads to a short exact sequence 
\[
\xymatrix
	{
	0 \ar[r]
 	& \Imr \lambda_{F} \ar[r]
 	& F \ar[r]
 	& \underline{F} \ar[r]
 	& 0,
	}
\]
where $\Imr \lambda_{F}$, being a quotient of the right-exact functor
$L_{0}F$ is epi-preserving. In fact, $\Imr \lambda_{F}$ is the largest epi-preserving subfunctor of $F$.
 Let~$\mathscr{T}$ be the class of epi-preserving functors and $\mathscr{F}$ the class of projectively stable functors. Then, similar to Proposition~\ref{P:torsion}, we have that 
$(\mathscr{T},\mathscr{F})$ is a hereditary torsion theory in the (possibly, large) category of additive functors. Moreover, similar to 
Proposition~\ref{P:maps-to-mono-preserving}, we have the following. 
Let $F$ and $F'$ be additive functors and suppose that $F'$ is epi-preserving. If a natural transformation $ \gamma : F' \to F$  evaluates to an isomorphism on each projective, then 
 $\coker \gamma$ is the projective stabilization of~$F$.
\end{remark}

\section{Injective stabilization and (co)satellites}\label{satellites}

Before making the next series of observations, we recall the definition of satellites~\cite{CE}. Let $F : \mathscr{A} \to \mathscr{B}$ be an additive (covariant) functor between abelian categories and assume that~$\mathscr{A}$ has enough projectives. The (first) left satellite $S_{1}F$ of $F$ is a functor $\mathscr{A} \to \mathscr{B}$ which is computed as follows. If~$A$ is an object 
of~$\mathscr{A}$, choose a projective syzygy sequence $0 \to \Omega A \overset{\iota}\to P \to A \to 0$ of $A$ and set $S_{1}F(A) := \Ker F(\iota)$. The values of $S_{1}F$ on morphisms are defined in an obvious way, using liftings of maps along syzygy sequences. Higher-order left satellites $S_{n}F$ of $F$ are defined inductively. Thus left satellites vanish on projectives.

Right satellites of an additive functor are defined symmetrically. Again, let 
$F : \mathscr{A} \to \mathscr{B}$ be an additive (covariant) functor between abelian categories and assume that~$\mathscr{A}$ has enough injectives.
The (first) right satellite $S^{1}F$ of $F$ is a functor $\mathscr{A} \to \mathscr{B}$ which is computed as follows. If~$A$ is an object of $\mathscr{A}$, choose an injective cosyzygy sequence $0 \to A \to I \overset{\pi}\to \Sigma A \to 0$ and set 
$S^{1}F(A) := \Coker F(\pi)$. The values of $S^{1}F$ on morphisms are defined in an obvious way, using extentions of maps along cosyzygy sequences. Higher-order right satellites $S^{n}F$ of $F$ are defined inductively. Thus right satellites vanish on injectives.

For now, we continue to assume that the domain category has enough injectives.
Immediately from the definition of the right satellite and Proposition~\ref{P:inj-stable} we have
\begin{corollary}
For any additive functor $F$ and any integer $i \geq 1$, the iterated right satellite $S^{i}F$ is injectively stable.  \qed
\end{corollary}

\smallskip

Embedding $B$ in different injectives $I$ and $I'$, we have short exact sequences $0 \longrightarrow B \longrightarrow I \longrightarrow \Sigma B \longrightarrow 0$ and $0 \longrightarrow B \longrightarrow I' \longrightarrow \Sigma' B \longrightarrow 0$. Extending the identity map on $B$, we have a commutative diagram 
\[
\xymatrix
	{
	0 \ar[r] 
	& B \ar[r] \ar@{=}[d]
	& I \ar[r] \ar[d]
	& \Sigma B \ar[r] \ar[d]^{\beta}
	& 0
\\
	0 \ar[r] 
	& B \ar[r] 
	& I' \ar[r]
	& \Sigma' B \ar[r]
	& 0
	}
\] 
By Schanuel's lemma, $\beta$ is part of an isomorphism $\Sigma B \amalg I' \longrightarrow \Sigma' B \amalg I$. Therefore, by Lemma~\ref{L:vanish-inj} and since 
$\overline{F}$ is additive, $\overline{F}(\beta)$ is an isomorphism. Moreover, any two choices for $\beta$ differ by a map factoring through an injective and, therefore, $\overline{F}(\beta)$ does not depend on the extension of the identity map on $B$. As a consequence, we have

\begin{lemma}\label{L:Ind-of-Sigma}
 Any two choices for $\overline{F}(\Sigma B)$ are canonically isomorphic. The canonical isomorphism is determined by any extension of the identity map on 
 $B$.  \qed
\end{lemma}

Using Proposition~\ref{P:inj-stable} and the definition of the right satellite, we can identify the group $\overline{F}(\Sigma B)$.

\begin{lemma}
If $F$ is an injectively stable functor, then $F(\Sigma \blank)$ is a functor and we have a functor isomorphism  $F(\Sigma \blank) \simeq S^{1}F(\blank)$. As a consequence, for any additive functor $F$ and any nonnegative integer $i$, we have a functor isomorphism 
$F(\Sigma^{i} \blank) \simeq  S^{i}F(\blank)$.
\end{lemma}

\begin{proof}
The desired isomorphism is an immediate consequence of the definition of the satellite and the fact that $F$ vanishes on injectives.
\end{proof}

For the next two propositions we assume that the domain 
category~$\mathscr{A}$ has enough injectives and projectives. In that case, the injective stabilization of a \texttt{half-exact} functor admits an alternative description. In~\cite{AB},  
%\marginpar{\fbox{Give a ref.}} 
Auslander and Bridger show that if~$F$ is half-exact, then $\overline{F} \simeq S^1S_1F$. Several years later, in \cite{FN}, Fisher-Palmquist and Newell establish that the satellites form an adjoint pair $(S^1,S_1)$ of endofunctors on the functor category.  Let $c:S^1S_1\to 1$ be the counit of adjunction. In~\cite{AB}, Auslander and Bridger show that the assignment $F\mapsto \overline{F}$ is functorial and is in fact the right adjoint to the inclusion functor of the injectively stable functors into the functor category, with the counit of adjunction\footnote{Of course, the authors did not know at the time that the satellites formed an adjoint pair! As we mentioned, that fact was established by Fisher-Palmquist and Newell several years later.} evaluated on $F$ being the inclusion $\alpha_F:\overline{F}\to F$. Combining these results we have

\begin{proposition}If $F$ is half-exact, then the $F$-component 
$c_F:S^1S_1F\to F$ of the counit of adjunction $c$ is isomorphic to the inclusion 
$\overline{F} \lra F$.
\end{proposition}

\begin{proof}
We have a commutative diagram with exact rows
\dia	{
	& S^1 S_1 F \ar[d]_{\theta}\ar[r]^{c_F}
	& F \ar@{=}[d]
\\
	0 \ar[r] 
	& \overline{F} \ar[d]_{\varphi}\ar[r]_{\alpha_F}
	& F \ar@{=}[d] \ar[r]_{\rho_{F}}
	& R^0 F
\\
	& S^1 S_1 F \ar[d]_{\theta} \ar[r]^{c_F} 
	& F \ar@{=}[d]
\\
	0 \ar[r]
	& \overline{F}\ar[r]_{\alpha_F}
	& F\ar[r]_{\rho_{F}}
	& R^0 F
	}
where 
\begin{enumerate}
\item $\theta$ is induced by the universal property of the counit of adjunction $\alpha$ (per Proposition~\ref{P:torsion}) and the fact that $S^1S_1F$ is injectively stable.
\smallskip
\item $\varphi$ is induced by the universal property of the counit of adjunction $c_F$ and the fact that $\overline{F}\simeq S^1S_1F$ because $F$ is half-exact.
\end{enumerate}
Thus $c_F\varphi\theta=c_F$ and $\alpha_{F} \theta \varphi = \alpha_{F}$. By the universal property of counits, $\varphi\theta = 1$ and $\theta \varphi = 1$, which is the desired claim.
%
%Because both $c_F$ is the counit of adjunction and $c_F\varphi\theta=c_F$ the morphisms $\varphi\theta=1$ because $c_F1=c_F$ and $c_F\varphi\theta=c_F$ and $\varphi\theta$ must be the unique morphism filling the diagram. Similarly because $\alpha_F$ is the counit of adjunction $\theta \varphi=1$. It follows that $\theta$ and $\varphi$ are both isomorphisms from which one deduces that $c_F$ is a kernel of the morphism $F\to R^0 F$.
\end{proof}

%\fbox{Adjunction of satellites and inj. stabilization as the counit of the adjunction.}

In the case when $F$ is not half-exact, we can only claim that the counit of adjunction factors through the injective stabilization. More precisely, we have

\begin{proposition}\label{P:dia-gen}For any additive functor $F$ there is a commutative diagram with exact rows and columns  
\dia
	{
	& 
	& 
	& 0 \ar[d] 
\\
	& 
	& 0 \ar[d] 
	& \overline{\Coker (c_F)} \ar@{..>}[d] 
\\
	& S^1 S_1 F \ar@{..>}[d]_\theta \ar[r]^{c_F}
	& F \ar[r] \ar@{=}[d] 
	& \Coker(c_F) \ar@{..>}[d]^{\gamma_F} \ar[r] 
	& 0 
\\
	0 \ar[r]
	& \overline{F} \ar@{..>}[d] \ar[r]_{\alpha_F}
	& F \ar[d] \ar[r]_{\rho_F}
	& R^0F
\\
	& \overline{\Coker(c_F)} \ar[d] 
	& 0
\\
	& 0
	}
%\begin{center}\begin{tikzpicture}
%\matrix (m) [ampersand replacement= \&,matrix of math nodes, row sep=2em, 
%column sep=2.5em,text height=1.5ex,text depth=0.25ex] 
%{\&\&\&0\\
%\&\&0\&\overline{\textrm{Coker}(c_F)}\\
%\&S^1S_1F\&F\&\textrm{Coker}(c_F)\&0\\
%0\&\overline{F}\&F\&R^0F\\
%\&\overline{\text{Coker}(c_F)}\&0\\
%\&0\\}; 
%\path[->,thick, font=\scriptsize]
%(m-3-2)edge node[above]{$c_F$}(m-3-3)
%%(m-1-2) edge[dashed] node[right]{$\varphi$}(m-2-1)
%(m-3-2) edge node[left]{$\varphi$}(m-4-2)
%(m-4-2) edge node{}(m-5-2)
%(m-5-2) edge node{}(m-6-2)
%(m-3-2) edge node{}(m-3-3)
%(m-3-3) edge node{}(m-3-4)
%(m-3-4) edge node{}(m-3-5)
%(m-2-3) edge node{}(m-3-3)
%(m-3-3) edge node[left]{$1$}(m-4-3)
%(m-4-1) edge node{}(m-4-2)
%(m-4-2) edge node{}(m-4-3)
%(m-4-3) edge node{}(m-4-4)
%(m-4-3) edge node{}(m-5-3)
%(m-4-2) edge node[below]{$\alpha_F$}(m-4-3)
%(m-1-4) edge node{}(m-2-4)
%(m-2-4) edge node{}(m-3-4)
%(m-3-4) edge node[right]{$\gamma_F$}(m-4-4);
%\end{tikzpicture}\end{center}
\end{proposition}

\begin{proof}The functor $S^1S_1F$ is injectively stable and hence, by the universal property of the counit $\alpha$, one has the morphism $\theta$ and the induced morphism~$\gamma_F$.  By the exactness of the two middle rows, $\gamma_F$ is an isomorphism on injectives.  By Proposition~\ref{P:maps-to-mono-preserving}, 
$\overline{\Coker (c_F)}$ is the kernel of $\gamma_F$. By the snake lemma, $\overline{\Coker (c_F)}$ is isomorphic to the cokernel of 
$\theta$.
\end{proof}

We now introduce new concepts, the \texttt{cosatellites} of a functor. Let 
$F : \mathscr{A} \to \mathscr{B}$ be an additive covariant functor between abelian categories. Assuming that~$\mathscr{A}$ has enough projectives, choose a projective syzygy sequence 
\[
0 \to \Omega A \overset{i}\to P \to A \to 0
\] 
for each object $A$ of $\mathscr{A}$. 

\begin{definition}
 The left cosatellite $C_{1}F$ of $F$ is the functor $\mathscr{A} \to \mathscr{B}$ whose value on~$A$ is defined by setting $C_{1}F(A) := \Coker F(i)$. The values on morphisms are defined in an obvious way.
\end{definition}

Next, instead of assuming that~$\mathscr{A}$ has enough projectives, assume that~$\mathscr{A}$ has enough injectives. Choose an injective cosyzygy sequence $0 \to A \to I \overset{p}\to \Sigma A \to 0$ for each object $A$ of $\mathscr{A}$. 

\begin{definition}
The right cosatellite $C^1F$ of $F$ is the functor 
$\mathscr{A} \to \mathscr{B}$ whose value on $A$ is defined by setting $C^{1}F(A) := \Ker F(p)$. The values on morphisms are defined in an obvious way. 
\end{definition}

\begin{definition}
If $F$ is a contravariant functor, then, under the above assumptions and notation, the left cosatellite $C_{1}F$ is defined by setting
$C_{1}F(A) := \Coker F(p)$ and the right cosatellite $C^{1}F$ is defined by setting $C^{1}F(A) := \Ker F(i)$. 
\end{definition}

In keeping with the philosophy of ~\cite{AB}, we mention some basic properties of  cosatellites related to injective stabilization. Just as the satellites determine the injective stabilization of a half-exact functor, the right cosatellite determines the image of the unit of adjunction $F\to R^0F$ for a half-exact functor.

\begin{proposition}\label{P:right-cosatellite}
Suppose that $F : \mathscr{A} \to \mathscr{B}$ is a half-exact functor. Then $C^1F$ is the image of $\rho_F : F\to R^0F$.
\end{proposition}

\begin{proof}
The cosyzygy sequence $0 \to A \overset{\iota}\to I \overset{p}\to \Sigma A \to 0$ and the fact that $F$ is half-exact, we have, upon applying Lemma~\ref{L:altinjdef}, the exact sequence 
\[
0\to\overline{F}(A)\to F(A)\overset{F(\iota)}\to F(I)\overset{F(p)}\to F(\Sigma A)
\]
From the defining short exact sequence 
\[
0\to \overline{F}(A)\to F(A)\to \Imr{\rho_F}(A)\to 0
\] 
it follows that $C^1F(A)=\Ker{F(p)} \simeq \Imr{\rho_F}(A)$ which is easily seen to be natural in $A$.  Hence $C^1F \simeq \Imr{\rho_F}$.
\end{proof}

\begin{remark}
 One has a dual result for the zeroth left-derived functors.  Namely, for any half-exact functor $F$, the left cosatellite $C_1F$ is isomorphic to the image of  $\lambda_{F} : L_0F\to F$.
\end{remark}

\section{The injective stabilization of a finitely presented functor}\label{inj-fp}

In this section we shall review the results from~\cite{A66} (see Sections 3 and 4) that allow to produce an explicit projective resolution of the injective stabilization of a finitely presented functor.

%Recall that the functor $F$ is \texttt{finitely presented} if it is isomorphic to the cokernel of a natural transformation between representable functors. In other words, there is a morphism $f : A \lra B$ in the domain category and an exact sequence
%
%
Given a morphism $f : A \lra B$, let
\[
\xymatrix
	{
	(B,\blank) \ar[r]^{(f,\blank)}
	& (A,\blank) \ar[r]
	& F \ar[r]
	& 0
	}
\] 
be a defining sequence for the covariant functor $F$. Of interest to us is the \texttt{defect}\label{Page:defect} $w(F) : = \Ker f$ of 
$F$.\footnote{For the original definition and a more detailed study of the defect, see~\cite{A66}. Notice that the term ``defect'' is not used there. For further properties and applications of the defect, see~\cite{R}.} Thus we have an exact sequence 
\begin{equation}\label{defect}
\begin{gathered}
\xymatrix
	{
	0 \ar[r]
	& w(F) \ar[r]^{l}
	& A \ar[rr]^{f} \ar@{->>}[rd]^{p}
	&
	& B \ar[r]
	& C \ar[r]
	& 0
\\
	& 
	&
	& \Imr f \ar@{>->}[ru]^{i}
	}
\end{gathered}
\end{equation}
where $(p,i)$ is the epi-mono factorization of $f$ and $C := \Coker f$. Associated with this sequence we have a diagram of solid arrows
\begin{equation}\label{4-term}
\begin{gathered}
  \xymatrix
	{
	&
	&
	& 0 \ar[d]
	& 0 \ar[d]
\\
	0 \ar[r]
	& (C, \blank) \ar[r] \ar@{=}[d]
	& (B,\blank) \ar[r]^{(i, \blank)} \ar@{=}[d]
	& (\Imr f,\blank) \ar[r] \ar[d]^{(p, \blank)} \ar@{}[dr] |*+[o][F-] {T}
	& F_{0} \ar[d]^{\nu} \ar[r]
	& 0
\\
	0 \ar[r]
	& (C, \blank) \ar[r]
	& (B,\blank) \ar[r]^{(f,\blank)} 
	& (A,\blank) \ar[r] \ar[d]^{(l, \blank)} \ar@{}[dr] |*+[o][F-] {M}
	& F \ar[r] \ar@{..>}[d]^{\mu}
	& 0
\\
	&
	&
	& (w(F), \blank) \ar@{:}[r] \ar[d]
	& (w(F), \blank) \ar@{..>}[d]
\\
	&
	&
	& F_{1} \ar@{:}[r] \ar[d]
	& F_{1} \ar[d]
\\
	&
	&
	& 0
	& 0
	}
\end{gathered}
\end{equation}
where:
\begin{itemize}
 \item $F_{0} := \Coker (i, \blank)$;
 \smallskip
 \item $\nu : F_{0} \to F_{1}$ is the induced map on the cokernels;
 \smallskip
 \item $F_{1} := \Coker (l, \blank)$.
% 
% \item $\nu : F \to (w(F), \blank)$ is the induced map from the cokernel;
% 
% 
\end{itemize}
Since the top two rows are exact, the square $T$ is both a pullback and a pushout. In particular, the induced map $\nu : F_{0} \to F$ is monic. Furthermore, by the pushout property of the square, the induced map between the cokernels of $(p, \blank)$ and $\nu$ is an isomorphism. Pushing out that map along the monomorphism part of the epi-mono factorization of $(l, \blank)$, we define 
$\mu : F \lra (w(F), \blank)$ as the composition of the cokernel of $\nu$ and the arrow parallel to the aforementioned monomorphism. Here we used the fact that an arrow parallel to an isomorphism is an isomorphism. Thus, we have the bottom isomorphism in the square $M$. Passing to the cokernels, we have the isomorphism on the bottom of the diagram. As a result, we have a commutative diagram with exact rows and columns.

As we saw in Proposition~\ref{P:rho-defect}, 
%Since a contravariant Hom functor with an injective argument is exact, both $F_{0}$ and $F_{1}$ vanish on injectives and are therefore injectively stable. In particular, 
$\mu : F \to \big(w(F),\blank \big)$ 
%is an isomorphism on injectives. Together with the fact that $\big(w(F),\blank \big)$ is left-exact, this shows that $\mu$ 
is isomorphic to the canonical morphism $\rho : F \to R^0F$.  
Moreover, since the kernel and the cokernel of any natural transformation between finitely presented functors are finitely presented (\cite[Proposition~3.1]{A82}), both $F_{0}$ and $F_{1}$ are finitely presented. The next result provides explicit projective resolutions of $F_{0}$ and $F_{1}$.

%In \cite{A66}, Auslander argues\footnote{See the argument at the bottom of page 209 and the beginning of page 210.} that 
%$\big(w(F),\blank \big) \simeq R^0F$\label{R0} for any finitely presented functor~$F$, and that the morphism $\mu : F \to\big(w(F),\blank\big)$ from~\eqref{4-term} is the canonical morphism $F\to R^0F$.  As a result, we have

\begin{proposition}\label{P:p-resolution-of-inj-stab}
Let $F$ be a finitely presented covariant functor. Then:
\begin{enumerate}
\item $\overline{F} = F_{0}$;
\smallskip
\item $F$ is injectively stable if and only if its defect $w(F)$ is zero;
\smallskip
\item given a presentation 
\dia
	{
	(B,\blank) \ar[r]^{(f,\blank)} 
	& (A,\blank) \ar[r]
	& F \ar[r]
	& 0,
	}   
$\overline{F} = F_{0}$ has a projective resolution 
\dia
	{
	0 \ar[r]
	& (C,-) \ar[r]
	& (B,\blank) \ar[r]^{(i, \blank)} 
	& (\Imr f, \blank) \ar[r] 
	& F_{0} \ar[r]
	& 0,
	} 
where $C := \Coker f$ and $i$ is the monic part of the epi-mono factorization $(p,i)$ of $f : A \to B$;
\smallskip
\item the exact sequence 
\[
\xymatrix
	{
	0 \ar[r]
	& w(F) \ar[r]^{l}
	& A \ar[r]^{p} 
	& \Imr f \ar[r]
	& 0
	}
\] 
gives rise to a projective resolution of $F_{1}$:
\[
\xymatrix
	{
	0 \ar[r]
	& (\Imr f, -) \ar[r] ^{(p,-)}
	& (A , -) \ar[r]^{(l,-)}
	& (w(F), -) \ar[r]
	& F_{1} \ar[r]
	& 0
	}
\] 
\end{enumerate}
\end{proposition}

\section{Exactness properties of the injective stabilization}\label{exact}

In this subsection we shall examine various exactness properties of the injective stabilization.

%The following result was proved in~\cite{AB} under the assumption that, in addition to enough injectives, the domain category has enough projectives. The elementary proof below shows that that additional assumption is not needed.

%The injective stabilization of a half-exact functor admits an alternative description. In~\cite{AB}, Auslander \marginpar{\fbox{Give a ref.}} and Bridger show that if $F$ is a half-exact functor, then $\overline{F} \simeq S^1S_1F$. On the other hand~\cite{CE}, any satellite of a half-exact functor is itself half-exact.  This yields

\begin{lemma}~\cite[Remark (1.9)]{AB}\label{L:half-exact}
 The injective stabilization of a half-exact functor is half-exact.
\end{lemma}

\begin{proof}
 Let $0 \to A \to B \to C \to 0$ be exact. Applying the horseshoe lemma to  injective containers of $A$ and $C$, followed by $F$ results in a commutative diagram with exact bottom two rows and columns:
 \[
\xymatrix
	{
	& 0 \ar[d]
	& 0 \ar[d]
	& 0 \ar[d]
\\
	& \overline{F}(A) \ar[d] \ar[r]
	& \overline{F}(B) \ar[d] \ar[r]
	& \overline{F}(C) \ar[d]
\\
	& {F(A)} \ar[d] \ar[r]
	& {F(B)} \ar[d] \ar[r]
	& {F(C)} \ar[d]
\\
	0 \ar[r]
	& {F(I')} \ar[r]
	& {F(I)}  \ar[r]
	& {F(I'')} \ar[r]
	& 0
	}
\] 
 The exactness of the top row now follows from a simple diagram chase.
\end{proof}

We can also easily characterize those functors whose injective stabilization is left-exact.

\begin{proposition}\label{P:left-exact}
Let $F$ be an additive functor. The following conditions are equivalent:
\begin{enumerate}
 \item $\overline{F}$ is left-exact;
 \smallskip
 \item $\overline{F} = 0$;
 \smallskip
 \item $F$ preserves monomorphisms.
\end{enumerate}
\end{proposition}

\begin{proof}
 Suppose $\overline{F}$ is left-exact. For an arbitrary object $L$ choose an injective container $\iota : L \longrightarrow I$ and apply $\overline{F}$. Then 
 $\overline{F}(\iota)$ is monic but, by Proposition~\ref{P:inj-stable},  $\overline{F}(I) = \{0\}$. Hence $\overline{F}(L) = \{0\}$. The converse is trivial. The equivalence of the last two conditions was established in Lemma~\ref{L:mono}. 
\end{proof}

Let ${\Lambda}$ be a ring and $A$ a right $\Lambda$-module. The injective stabilization of the functor $A \otimes \blank$ will be denoted by 
$A \ot \blank$.\label{harpoon} Thus $A \ot B = (A \ot \blank)(B)$. In this context, $A$ will be said to be the \texttt{inert} variable, and $B$ will be referred to as the \texttt{active} variable. In other words, the active variable is the one being injectively resolved. In the new notation the harpoon always points to the active variable and thus $A \ot B$ will not be confused with $A \otleft B$. That these two values could be different can be seen from

\begin{example}

Take $\Lambda : = \Z$.  Then:

\begin{itemize}
 \item $\Z \ot \Q/\Z = 0$ (because $\Q/\Z$ is injective);
 \smallskip
 \item $\Z \overset{\leftharpoondown}{\otimes} \Q/\Z = \Q/\Z$
 (just tensor $0 \to \Z \to \Q$ with $\Q/\Z$).
\end{itemize}

\end{example}

Specializing Proposition~\ref{P:left-exact} to the case of the tensor product, we have

\begin{proposition}
The functor $A \ot \blank$ is left-exact if and only if $A$ is flat. In that case, $A \ot \blank = 0$. \qed
\end{proposition}

Next we want to investigate the question of when the injective stabilization of an additive functor is right-exact. The following example shows that, in general, the injective stabilization of an additive functor need not even preserve epimorphisms. 

\begin{example}\label{E:not-right-exact}
Let $\Bbbk$ be a field, $\Lambda := \Bbbk [x]/(x^{2})$, and
 $F := \Bbbk \otimes \blank$. Then $\Lambda$ is self-injective and therefore $\overline{F}(\Lambda) = \{0\}$. The canonical epimorphism 
$\pi : \Lambda \longrightarrow \Bbbk$ yields a short exact sequence 
 \[
 0 \longrightarrow \Bbbk \overset{\iota}\longrightarrow \Lambda \overset{\pi}\longrightarrow
 \Bbbk \longrightarrow 0
 \]
 where $\iota : \Bbbk \longrightarrow \Lambda$ is the injective envelope. Tensoring it with $\Bbbk$, we have an exact sequence
 \[
\Bbbk \overset{1 \otimes \iota}{\longrightarrow} \Bbbk \overset{1 \otimes \pi}{\longrightarrow} \Bbbk \longrightarrow 0
 \]
and thus $1 \otimes \pi$ is an isomorphism. This forces $1 \otimes \iota = 0$, showing that 
$\overline{F}(\Bbbk) \simeq \Bbbk$. Thus $\overline{F}(\pi) : 0 \to \Bbbk$ is not epic and therefore $\overline{F}$ does not preserve epimorphisms.
  
%\footnote{For a more conceptual approach to this example, see Example~\ref{E:qF} below.}
\end{example}

To deal with this kind of obstruction, we introduce

\begin{definition}\label{D:propA=FPinj}
Let $ F : \mathscr{C} \longrightarrow \mathscr{D}$ be an additive covariant functor between abelian categories, where $\mathscr{C}$ has enough injectives. We say that an object $C \in \mathscr{C}$ has property~\texttt{A} with respect to $F$ if $F(i)$ is a monomorphism whenever $i$ is a monomorphism with domain~$C$.
\end{definition}

\begin{remark}\label{R:abs-pure}
If $\mathscr{C}$ is a module category and $C$ has property \texttt{A} with respect to any tensor product functor, then $C$ is said to be absolutely pure~(\cite{Mad}) or $FP$-injective~(\cite{S}).
\end{remark}

Trivially, any injective object has property \texttt{A}.

\begin{lemma}\label{L:mono-to-inj}
 Under the above assumptions, $C$ has property \texttt{A} with respect to ${F}$
 if and only if  $F(i) : F(C) \to F(J)$ is a monomorphism for each monomorphism $i : C \to J$ with $J$ injective. 
\end{lemma}

\begin{proof}
 The ``only if'' part is immediate from the definition. For the other direction, use the same argument as in the proof of Lemma~\ref{L:mono}.\footnote{Clearly, it suffices to check the preservation property for a single injective container of $C$.}
\end{proof}

\begin{proposition}\label{P:epi-A}
Under the above assumptions, suppose that $\overline{F}$ preserves epimorphisms. Then all cosyzygy objects (i.e., homomorphic images of injectives\footnote{For modules over a commutative domain with this property, Matlis~\cite{Mat} uses the term $h$-\texttt{divisible}. In~\cite{GD}, the same terminology is used for modules over arbitrary rings.}) have property \texttt{A}
with respect to $F$.
\end{proposition}

\begin{proof}
 Given an arbitrary object $C$, let $0 \to C \to I \overset{p}\to \Sigma C \to 0$ be exact with $C$ injective. Applying $\overline{F}$, we have that $\overline{F}(p)$ is epic and $\overline{F}(I) = (0)$, showing that $\overline{F}(\Sigma C) = (0)$. In other words, the kernel of $F(i) : F(\Sigma C) \to F(J)$ is zero for any injective container $i : \Sigma C \to J$. By Lemma~\ref{L:mono-to-inj}, $\Sigma C$ has property \texttt{A}.
\end{proof}

Now we want to show that if $F$ is right-exact, then the converse of the previous result holds. More precisely, we have

\begin{theorem}
 Let $ F : \mathscr{C} \longrightarrow \mathscr{D}$ be a right-exact (hence additive) functor between abelian categories, where $\mathscr{C}$ has enough injectives. Then  $\overline{F}$ is right-exact if and only if all cosyzygy objects in $\mathscr{C}$ have property \texttt{A} with respect to $F$.
\end{theorem}

\begin{proof}
The ``only if'' part is Proposition~\ref{P:epi-A}. Conversely, assume that 
 property~\texttt{A} holds for all cosyzygy modules. By Lemma~\ref{L:half-exact},  $\overline{F}$ is half-exact. Thus we only need to show that  $\overline{F}$ preserves epimorphisms. Let 
\[
0 \longrightarrow L \overset{i}{\longrightarrow} M \overset{p}{\longrightarrow} N \longrightarrow 0
\]
be a short exact sequence. We need to show that $\overline{F}(p)$ is epic. Taking cosyzygy sequences for the end terms and using the horseshoe lemma, we have a $3 \times 3$ commutative diagram with exact rows and columns, whose middle row is split-exact. Applying $F$ and using its right-exactness, we have another commutative diagram with exact rows and columns:

\begin{equation}\label{first-square}
\begin{gathered}
 \xymatrix
	{
	& \overline{F}(L) \ar[r] \ar@{^{(}->}[d] %\ar@{} [dr] |{T}
	& \overline{F}(M) \ar[r]^{\overline{F}(p)} \ar@{^{(}->}[d]
	& \overline{F}(N)\ar@{^{(}->}[d]
	&
\\
	& F(L) \ar[r] \ar[d]
	& F(M) \ar[r] \ar[d]
	& F(N) \ar[r] \ar[d]
	& 0
\\
	0 \ar[r]
	& F(I') \ar[r] \ar@{->>}[d]
	& F(I) \ar[r] \ar@{->>}[d]
	& F(I'') \ar[r] \ar@{->>}[d]
	& 0
\\
	& F(\Sigma L) \ar[r]^{F(\Sigma i)}
	& F(\Sigma M) \ar[r]
	& F(\Sigma N) \ar[r]
	& 0
	}
\end{gathered}
\end{equation}
The snake lemma yields an exact sequence
\[
\xymatrix
	{
	\overline{F}(M) \ar[r]^{\overline{F}(p)} 
	& \overline{F}(N) \ar[r]^{\delta}
	& F(\Sigma L) \ar[r]^{F(\Sigma i)}
	& F(\Sigma M)
	}
\]
where $\delta$ is the connecting homomorphism. By the assumption, 
$F(\Sigma i)$ is monic, making $\delta$  the zero map and 
$\overline{F}(p)$ epic.
  \end{proof}

Specializing to the case $F := A \ot \blank$, where $A$ is a right 
$\Lambda$-module, we can give a criterion for $\Lambda$ to have the property that the functor $A \ot \blank$ is right-exact \texttt{for any} right $\Lambda$-module~$A$.

\begin{theorem}
The functor $A \ot \blank$ is right-exact for any right $\Lambda$-module $A$ if and only if the class of all cosyzygy modules of left $\Lambda$-modules has property~\texttt{A} with respect to all tensor product functors.  \qed
\end{theorem}

\begin{remark}
 In the terminology of~\cite{Mat} and~\cite{S}, the last result reads as follows: the functor $A \ot \blank$ is right-exact for any right 
 $\Lambda$-module $A$ if and only if all $h$-divisible left $\Lambda$-modules are $FP$-injective. Concerning the terminology, see Remark~\ref{R:abs-pure} and the footnote to Proposition~\ref{P:epi-A}.
\end{remark}

%\begin{proof}
%Suppose that $A \ot \blank$ is right-exact for any $A$ and let $\Sigma C$ be a cosyzygy module of some module $C$. Thus we have a short exact sequence
%\[
%0 \longrightarrow C \longrightarrow I \overset{p}{\longrightarrow} \Sigma C\longrightarrow 0
%\]
%where $I$ is injective. By assumption, $A \ot p$ is epic, but by 
%Proposition~\ref{P:inj-stable}, $A \ot I \simeq \{0\}$, and therefore 
%$A \ot \Sigma C \simeq \{0\}$. The latter means that for any monomorphism $\iota : \Sigma C \to J$, where $J$ is injective, 
%$A \otimes \iota$ is monic. By Lemma~\ref{L:mono-to-inj}, we have the desired property PI.
%
%\end{proof}

\begin{corollary}
 If $\Lambda$ is left (resp., right) semihereditary, then $A \ot \blank$ (resp., $\blank \otleft A$) is right-exact for any right (resp., left)  $\Lambda$-module $A$. 
\end{corollary}

\begin{proof}
 By~\cite[Theorem 2]{Meg}, a ring is semihereditary if and only if any homomorphic image of an $FP$-injective is $FP$-injective. Since injectives are clearly $FP$-injective, any $h$-divisible module is $FP$-injective.
\end{proof}

\section{The injective stabilization of a right-exact functor}\label{right-exact}

In this section, we assume that the domain category has enough injectives and projectives. Let $F : \Lambda\textrm{-}\mathrm{Mod} \longrightarrow
\mathrm{Ab}$ be a right-exact functor. Thus $F$ is automatically additive. 

%==================

Applying $F$ to the short exact sequence 
\[
0 \longrightarrow B \longrightarrow I \longrightarrow \Sigma B \longrightarrow 0
\]
with an injective $I$ and passing to the corresponding long homology exact sequence

\begin{equation}\label{SL to inj. stab}
\begin{gathered}
 \xymatrix
	{
	\ldots \ar[r]
	& L_{1}F(I) \ar[r]
	& L_{1}F(\Sigma B) \ar[r]^{\delta} \ar@{->>}[d]
	& F(B) \ar[r]
	& F(I) \ar[r]
	& \ldots
\\
	&
	& S^{1}L_{1}F(B) \ar[r]^>>>>>{\simeq}
	& \overline{F}(B) \ar@{>->}[u]
	& 
	&
	}
\end{gathered}
\end{equation}
%\[
%\xymatrix
%	{
%	\Tor_{1}(A,I) \ar[r]
%	& \Tor_{1}(A, \Sigma B) \ar[rr]^{\delta} \ar@{->>}[d]
%	&
%	& A \otimes B \ar[r]
%	& A \otimes \Sigma B
%\\
%	& S^{1}\Tor_{1}(A, \blank)(B) \ar[rr]^>>>>>>>>>>{\simeq}
%	&
%	& A \otB \ar@{>->}[u]
%	& 
%	}
%\] 
we have a canonical isomorphism $S^{1}L_{1}F(B) \longrightarrow \overline{F}(B)$ induced by the connecting homomorphism $\delta$.\footnote{We assume that $\delta$ is chosen canonically, as in the classical proof of the snake lemma.} As a result, in the case of a right-exact functor $F$, we have yet another description of the injective stabilization.

\begin{lemma}\label{L:SL}
For a right-exact functor $F$, the connecting homomorphism 
$\delta$ in \eqref{SL to inj. stab} induces a functor isomorphism $S^{1}(L_{1}F) \simeq \overline{F}$. 
\end{lemma}

\begin{proof}
The componentwise isomorphisms have just been constructed. The naturality  follows from the naturality of $\delta$ and from a trivial diagram chase.
\end{proof}

\begin{remark}
Let 
\[
0 \longrightarrow B \longrightarrow I' \longrightarrow \Sigma' B \longrightarrow 0
\] 
be an embedding of $B$ in another injective $I'$. Extending the identity map on $B$ to a map of the corresponding short exact sequences and using the just proved naturality, we have a commutative diagram of isomorphisms
 \[
\xymatrix
	{
	S^{1}L_{1}F(B) \ar[r]^>>>>>{\cong} \ar[d]^{\cong}
	& \overline{F}(B) \ar[d]^{\cong}
\\
	S^{1'}L_{1}F(B) \ar[r]^>>>>>{\cong}
	& \overline{F}'(B) 	
	}
\] 
Colloquially, we shall simply say that the isomorphism of Lemma~\ref{L:SL} is  determined uniquely up to a canonical isomorphism.
\end{remark}

%==================
Suppose that
\[
0 \longrightarrow B' \longrightarrow B \longrightarrow B'' \longrightarrow 0
\]
is an exact sequence of left $\Lambda$-modules. We want to construct a long exact sequence associated with it and with the injective stabilization of $F$. 

Choose injective resolutions $I'$ of $B'$ and $I''$ of $B''$, and using the horseshoe lemma, build an injective resolution of $B$. Passing to the first cosyzygy modules and applying the functor 
$F$ to the resulting $3 \times 3$ square with exact rows and columns, we have, because $F$ is right-exact,  another commutative diagram with exact rows and columns
\begin{equation}\label{second-square}
\begin{gathered}
 \xymatrix
	{
	& \overline{F}(B') \ar[r] \ar@{^{(}->}[d] %\ar@{} [dr] |{T}
	& \overline{F}(B) \ar[r] \ar@{^{(}->}[d]
	& \overline{F}(B'' )\ar@{^{(}->}[d]
	&
\\
	& F(B') \ar[r] \ar[d]
	& F(B) \ar[r] \ar[d]
	& F(B'') \ar[r] \ar[d]
	& 0
\\
	0 \ar[r]
	& F(I'^{0}) \ar[r] \ar@{->>}[d]
	& F(I^{0}) \ar[r] \ar@{->>}[d]
	& F(I''^{0}) \ar[r] \ar@{->>}[d]
	& 0
\\
	& F(\Sigma B') \ar[r]
	& F(\Sigma B) \ar[r]
	& F(\Sigma B'') \ar[r]
	& 0
	}
\end{gathered}
\end{equation}
whose second-from-the-bottom row is split-exact by the additivity of $F$. The snake lemma gives rise to an exact sequence
\[
\xymatrix
	{
	\overline{F}(B') \ar[r] 
	& \overline{F}(B) \ar[r] 
	& \overline{F}(B'' )\ar[r]
	& F(\Sigma B') \ar[r]
	& F(\Sigma B) \ar[r]
	& F(\Sigma B'')
	}
\] 

The same argument applied to the exact sequence 
\[
0 \longrightarrow \Sigma B' \longrightarrow \Sigma B \longrightarrow \Sigma B'' \longrightarrow 0
\]
yields a similar diagram 
\[
\xymatrix
	{
	& \overline{F}(\Sigma B') \ar[r] \ar@{^{(}->}[d]
	& \overline{F}(\Sigma B) \ar[r] \ar@{^{(}->}[d]
	& \overline{F}(\Sigma B'') \ar@{^{(}->}[d]
	&
\\
	& F(\Sigma B') \ar@{}[dr] |*+[o][F-] {S} \ar[r] \ar[d]
	& F(\Sigma B) \ar[r] \ar[d]
	& F(\Sigma B'') \ar[r] \ar[d]
	& 0
\\
	0 \ar[r]
	& F(I'^{1}) \ar[r] \ar@{->>}[d]
	& F(I^{1}) \ar[r] \ar@{->>}[d]
	& F(I''^{1}) \ar[r] \ar@{->>}[d]
	& 0
\\
	& F(\Sigma^{2} B') \ar[r]
	& F(\Sigma^{2} B) \ar[r]
	& F(\Sigma^{2} B'') \ar[r]
	& 0
	}
\]
and an exact sequence 
\[
\xymatrix@C15pt
	{
	\overline{F}(\Sigma B') \ar[r] 
	& \overline{F}(\Sigma B) \ar[r] 
	& \overline{F}(\Sigma B'') \ar[r]
	& F(\Sigma^{2} B') \ar[r]
	& F(\Sigma^{2} B) \ar[r]
	& F(\Sigma^{2} B'')
	}
\]
Taking into account that the square $S$ is commutative and its bottom map is monic, we have that the connecting homomorphism 
in~\eqref{second-square} has its image in $\overline{F}(\Sigma B')$, and thus gives rise to a homomorphism 
$\delta : \overline{F}(B'') \longrightarrow \overline{F}(\Sigma B')$. 
It is easy to see that, as a result, we have an exact sequence
\[
\xymatrix
	{
	\overline{F}(B') \ar[r] 
	& \overline{F}(B) \ar[r] 
	& \overline{F}(B'')\ar[r]^{\delta}
	& \overline{F}(\Sigma B') \ar[r]
	& \overline{F}(\Sigma B) \ar[r]
	& \overline{F}(\Sigma B'')
	}
\] 
Moreover, iterating the above procedure, we have 

\begin{lemma}\label{L:to-the-right}
 The just constructed sequence of injective stabilizations
 \[
 \xymatrix@C15pt
	{
	\overline{F}(B') \ar[r] 
	& \overline{F}(B) \ar[r] 
	& \ldots \ar[r]
	& \overline{F}(\Sigma^{i} B'') \ar[r]^{\delta}
	& \overline{F}(\Sigma^{i+1} B') \ar[r]
	& \overline{F}(\Sigma^{i+1} B) \ar[r]
	& \ldots
	}
\]
is exact and is natural with respect to morphisms of short exact sequences.
\end{lemma}
\begin{proof}
 The first assertion has already been established; the second follows from the functoriality of the injective stabilization and the naturality of the connecting homomorphism.
\end{proof}

\begin{proposition}
The injective stabilization of a right-exact functor is half-exact. If, in addition, the ring is hereditary, then the injective stabilization of the functor is also right-exact.
\end{proposition}

\begin{proof}
 The first claim follows from Lemma~\ref{L:half-exact}.
%{L:to-the-right}. 
 To prove the second claim, notice that first cosyzygy modules are injective and therefore the bottom row of the diagram~\eqref{second-square} is a split short exact sequence, making the connecting homomorphism a zero map.
\end{proof}

Now we want to extend the sequence from Lemma~\ref{L:to-the-right} to the left so that the resulting doubly-infinite sequence be exact. To this end, we again apply $F$ (which is still assumed to be right-exact) to the short exact sequence $0 \to B' \to B \to B'' \to 0$. This yields the familiar exact sequence
\[
\xymatrix@C10pt
	{
	L_{1}F(B') \ar[r]
	& L_{1}F(B) \ar[r]
	& L_{1}F(B'') \ar[r]^-{\delta}
	& F(B') \ar[r]^-{\beta}
	& F(B) \ar[r]
	& F(B'') \ar[r]
	& 0
	}
\] 
which extends on the left to the long exact sequence of the corresponding left-derived functors of $F$. On the other hand, the last three terms of this sequence are part of the commutative diagram~\eqref{second-square} with exact rows and columns. Taking into account that the map $F(I'^{0}) \longrightarrow F(I^{0})$ in that diagram is monic, we have that the image of the connecting homomorphism $\delta : L_{1}F(B'') \to F(B')$ is in 
$\overline{F}(B')$, which yields a sequence
\[
\xymatrix@C15pt
	{
	L_{1}F(B') \ar[r]
	& L_{1}F(B) \ar[r]
	& L_{1}F(B'') \ar[r]^-{\delta'}
	& \overline{F}(B') \ar[r]^{\beta'}
	& \overline{F}(B) \ar[r]
	& \overline{F}(B'')
	}
\]
It is clear that this sequence is a complex and that it is exact, except possibly at $\overline{F}(B')$. On the other hand, 
$\Ker \beta' \subset \Ker \beta = \Imr \delta = \Imr \delta'$ (the last three terms are indeed \texttt{equal}, not just isomorphic) and therefore we have the exactness at $\overline{F}(B')$, too. We have thus proved 

\begin{proposition}\label{P:long-exact}
 Given a right-exact functor $F$ from (left or right) $\Lambda$-modules to abelian groups and a short exact sequence 
$0 \to B' \to B \to B'' \to 0$ of (left or right) modules, the foregoing constructions yield a doubly-infinite exact sequence 
\[
\xymatrix@R10pt
	{
	\ldots \ar[r]
	& L_{i}F(B') \ar[r]
	& L_{i}F(B) \ar[r]
	& L_{i}F(B'') 
	&
\\
	& \vdots
	& \vdots
	& \vdots
	&	
\\
	\ldots \ar[r]
	& L_{1}F(B') \ar[r]
	& L_{1}F(B) \ar[r]
	& L_{1}F(B'')
	&
\\
	\ar[r]
	& \overline{F}(B') \ar[r]
	& \overline{F}(B) \ar[r]
	& \overline{F}(B'')
	&
\\
	\ar[r]
	& \overline{F}(\Sigma B') \ar[r]
	& \overline{F}(\Sigma B) \ar[r]
	& \overline{F}(\Sigma B'')
	&
\\
	& \vdots
	& \vdots
	& \vdots
	&
\\
	\ldots \ar[r]
	& \overline{F}(\Sigma^{j}B') \ar[r]
	& \overline{F}(\Sigma^{j}B) \ar[r]
	& \overline{F}(\Sigma^{j}B'') \ar[r]
	& \ldots
	}
\]  \qed
\end{proposition}

\section{The injective stabilization of the tensor product}\label{inj-tensor-prod}

Now we specialize to the case when $F$ is a univariate tensor product on a module category. More precisely, given a ring $\Lambda$ and a right 
$\Lambda$-module $A$, we are interested in the injective stabilization 
$A\, \ot \blank$ of the functor $A \otimes \blank : \Lambda\textrm{-}\mathrm{Mod} \longrightarrow \mathrm{Ab}$ (see p.~\pageref{harpoon} for notation.)

If $P$ is a projective right module then, by Corollary~\ref{C:tensor-flat}, $P\, \ot\, \blank$ is a zero functor. This, together with Schanuel's lemma, yields

\begin{lemma}\label{L:Omega}
The operation $\Omega A\, \ot\, \blank$, where $\Omega A$ denotes a first syzygy module of~$A$, is determined uniquely up to a canonical isomorphism, regardless of the choice of the projective resolution of  $A$. The canonical isomorphism is induced by any comparison map between the projective resolutions of $A$. This operation is a functor. \qed
\end{lemma}

Similarly, specializing Lemma~\ref{L:Ind-of-Sigma} to  $F := A\, {\otimes} \blank$, we have

\begin{lemma}\label{L:Sigma}
The operation $A\, \ot\, \Sigma \blank$, where $\Sigma$ denotes a cosyzygy operation on left $\Lambda$-modules in injective resolutions, is determined uniquely up to a canonical isomorphism, which is induced by any comparison map between the injective resolutions of the blank argument. This operation is a functor. \qed
\end{lemma}

Specializing~\eqref{SL to inj. stab} and Lemma~\ref{L:SL} to the case $F = A \otimes \blank$, we have 

\begin{proposition}\label{P:via-Tor}
The connecting homomorphism 
$\Tor_{1}(A, \Sigma B) \longrightarrow A \otimes B$ induces a functor isomorphism $S^{1} \Tor_{1}(A, \blank) \simeq A \, \ot \, \blank$, 
natural in $A$.\footnote{The reader who does not have prior experience in dealing with functors should notice that, when evaluating the left-hand side of the last isomorphism on the module $B$, one cannot replace the blank with $B$ -- the resulting expression $S^{1}\Tor_{1}(A, B)$ would be meaningless. Instead, it is useful to think of $S^{1}$ as a derivative, and follow the freshman calculus rule to compute the derivative first and then evaluate it at a specific value of the argument.} \qed
\end{proposition}

Specializing Proposition~\ref{P:long-exact} to the case 
$F : = A \otimes \blank$, where $A$ is a right $\Lambda$-module, and a short exact sequence $0 \to B' \to B \to B'' \to 0$ of left  modules, we have a doubly-infinite exact sequence
\begin{equation}\label{long}
\begin{gathered}
\xymatrix@R10pt
	{
	\ldots \ar[r]
	& \Tor_{i}(A,B') \ar[r]
	& \Tor_{i}(A,B) \ar[r]
	& \Tor_{i}(A,B'') 
	&
\\
	& \vdots
	& \vdots
	& \vdots
	&	
\\
	\ldots \ar[r]
	& \Tor_{1}(A,B') \ar[r]
	& \Tor_{1}(A,B) \ar[r]
	& \Tor_{1}(A,B'')
	&
\\
	\ar[r]
	& A \ot B' \ar[r]
	& A \ot B \ar[r]
	& A \ot B''
	&
\\
	\ar[r]
	& A \ot \Sigma B' \ar[r]
	& A \ot \Sigma B \ar[r]
	& A \ot \Sigma B''
	&
\\
	& \vdots
	& \vdots
	& \vdots
	&
\\
	\ldots \ar[r]
	& A \ot \Sigma^{j}B' \ar[r]
	& A \ot \Sigma^{j}B \ar[r]
	& A \ot \Sigma^{j}B'' \ar[r]
	& \ldots
	}
\end{gathered}
\end{equation}
\bigskip

Finally, we want to evaluate the functor $A \ot \blank$ on the bimodule $\Lambda$. Thus $A \ot \Lambda$ is a submodule of $A \otimes \Lambda \cong A$. If $A$ is finitely presented, this module can be determined explicitly. To this end, we specialize Proposition~\ref{P:p-resolution-of-inj-stab} to the $\Lambda$-component. Let $P_{1} \overset{\partial}\lra P_{0} \lra A \lra 0$ be a finite presentation of $A$. This yields a finite presentation of the functor $A \otimes \blank$:
\[
\xymatrix
	{
	P_{1} \otimes \blank \ar[r]^{(\partial \otimes \blank)} \ar@{=}[d]
	& P_{0} \otimes \blank \ar[r] \ar@{=}[d]
	& A \otimes \blank \ar[r] \ar@{=}[d]
	& 0
\\
	(P_{1}^{\ast}, \blank) \ar[r]^{(\partial^{\ast}, \blank)}
	& (P_{0}^{\ast}, \blank) \ar[r]
	& A \otimes \blank \ar[r]
	& 0
	}
\] 
This, in turn, yields the corresponding defect sequence
\begin{equation}\label{defect-tensor}
\begin{gathered}
\xymatrix
	{
	0 \ar[r]
	& A^{\ast} \ar[r]^{l}
	& P_{0}^{\ast} \ar[rr]^{\partial^{\ast}} \ar@{->>}[rd]^{p}
	&
	& P_{1}^{\ast} \ar[r]
	& \Tr A \ar[r]
	& 0
\\
	& 
	&
	& \Omega \Tr A  \ar@{>->}[ru]^{i}
	}
\end{gathered}
\end{equation}
Diagram~\eqref{4-term} now becomes
\begin{equation*}\label{4-term-tensor}
\begin{gathered}
  \xymatrix
	{
	&
	&
	& 0 \ar[d]
	& 0 \ar[d]
\\
	0 \ar[r]
	& (\Tr A, \blank) \ar[r] \ar@{=}[d]
	& (P_{1}^{\ast},\blank) \ar[r]^{(i,-)} \ar@{=}[d]
	& (\Omega \Tr A,\blank) \ar[r] \ar[d]^{(p,-)}
	& \Ext^{1}(\Tr A, \blank) \ar[d]^{\nu} \ar[r]
	& 0
\\
	0 \ar[r]
	& (\Tr A, \blank) \ar[r]
	& (P_{1}^{\ast},\blank) \ar[r]^{(\partial^{\ast},\blank)} 
	& (P_{0}^{\ast},\blank) \ar[r] \ar[d]^{(l,-)}
	& A \otimes \blank \ar[r] \ar[d]^{\mu}
	& 0
\\
	&
	&
	& (A^{\ast}, \blank) \ar@{=}[r] \ar[d]
	& (A^{\ast}, \blank) \ar[d]
\\
	&
	&
	& \Ext^{1}(\Omega \Tr A, \blank) \ar@{=}[r] \ar[d]
	& \Ext^{1}(\Omega \Tr A, \blank) \ar[d]
\\
	&
	&
	& 0
	& 0
	}
\end{gathered}
\end{equation*}
As a consequence, we have
\begin{proposition}[\cite{AB}, Corollary (2.9)]\label{P:4-term-tensor}
 If $A$ is a finitely presented right $\Lambda$-module, then 
\begin{equation}\label{D:A ot}
 A\, \ot \blank \simeq \Ext^{1}(\Tr A, \blank),
\end{equation}
with the top row of the preceding diagram giving a projective resolution of
$A\, \ot \blank$. Specializing the rightmost column to the 
$\Lambda$-component, we have $A \ot \Lambda = \Ker \mu_{\Lambda}$, where 
$\mu_{\Lambda} = e_{A} : A \to A^{\ast\ast}$ is the canonical evaluation map. \qed
\end{proposition}

\begin{remark}
 The reader may want to compare \eqref{D:A ot} with the formula 
 \[
 (\underline{A, \blank}) \simeq \Tor_{1}(\Tr A, \blank)
 \]
 expressing the projective stabilization of the covariant Hom functor in terms of Tor when $A$ is finitely presented.
\end{remark}

%The above diagram now allows to compute the right-derived functors of the univariate tensor product with a finitely presented module.
%
%\begin{proposition}
% Let $A$ be a finitely presented right $\Lambda$-module. Then
% \[
% R^{n}(A \otimes \blank) \simeq \Ext^{n}(A^{\ast}, \blank)
% \]
% for each $n \geq 0$.
%\end{proposition}
%
%\begin{proof}
% It is well-known that $R^{n}F \simeq R^{n}(R^{0}F)$ any additive functor $F$. As we mentioned on page~\pageref{R0}, if $F$ is finitely presented, then $R^{0}F$ is isomorphic to the covariant Hom functor whose contravariant argument is the defect of $F$. In the case when $F$ is the tensor product with a finitely presented module, the above diagram shows that the defect is the dual of the module. The claim now follows. 
%\end{proof}
%
%\begin{remark}
% The zeroth left-derived functor of the covariant Hom functor with a finitely presented contravariant argument $C$ was determined in (\cite[Proposition 7.1]{A66}):
%\[
%L_{0}(C, \blank) \simeq C^{\ast} \otimes \blank
%\]
%(Here $C$ is a left $\Lambda$-module). It now follows that
% \[
% L_{n} (C, \blank) \simeq \Tor_{n}(C^{\ast},\blank)
% \]
%\end{remark}

\begin{remark}\label{bifunctor}
For future use, we make a simple observation that the injective stabilization of the tensor product is a bifunctor. This follows from the fact that the tensor product is a bifunctor and a standard diagram chase.
\end{remark}

Finally, we want to examine the injective stabilization of the tensor product when $\Lambda$ is an algebra over a commutative ring $R$. Choose an 
injective $R$-module $\mathbf{J}$ and let $D_{\mathbf{J}} : = \Hom_{R}(\blank, \mathbf{J})$. Define the functor Hom modulo injectives by setting 
\[
\overline\Hom(X,Y) := \Hom(X,Y)/I(X,Y)
\]
where $I(X,Y)$ denotes the subgroup of all homomorphisms factoring through injective modules. We now have

\begin{proposition}\label{P:AR-formula}
 Let $A$ be a right $\Lambda$-module and $B$ a left $\Lambda$-module. There is an isomorphism 
 \[
 D_{\mathbf{J}}(A\ot B) \simeq \overline{\Hom} (B, D_{\mathbf{J}}(A)), 
 \]
functorial in $A$ and $B$.
\end{proposition}

\begin{proof}
 Let $0 \to B \to I$ be the injective envelope of $B$. Applying the exact 
 functor~$D$ to the defining sequence 
 \[
 0 \lrt A\ot B \lrt A\otimes B \lrt A\otimes I
 \]
 we have the exact sequence 
 \[
 (A\otimes I, \mathbf{J}) \lrt (A\otimes B, \mathbf{J}) \lrt
 (A\ot B, \mathbf{J}) \lrt 0
 \]
By the adjoint property of the tensor product and Hom, this rewrites as 
\[
(I, D_{\mathbf{J}}(A))  \lrt (B, D_{\mathbf{J}}(A)) \lrt  \overline{\Hom}(B, D_{\mathbf{J}}(A)) \lrt 0
\] 
Comparing the rightmost terms, we have the desired result.
\end{proof}

\begin{remark}
 Suppose now that the module $A$ from Proposition~\ref{P:AR-formula} is finitely presented. Then, up to projective equivalence, $A$ can be written as
 $\Tr A'$ for some left $\Lambda$-module $A'$. Since tensoring with a projective 
 is an exact functor, the value of the injective stabilization $(\Tr A') \ot B$ is well-defined. Also, since $D_{\mathbf{J}}$ converts projectives into injectives, the module $D_{\mathbf{J}}(\Tr A')$ is defined uniquely modulo injectives. Together with Proposition~\ref{P:4-term-tensor}, this yields a well-defined natural isomorphism 
 \[
 D_{\mathbf{J}}\Ext^{1}(A',B) \simeq \overline{\Hom} (B, D_{\mathbf{J}}\Tr A')
 \]
which is nothing but the Auslander-Reiten formula~\cite{AR-III}. Notice however that the Auslander-Reiten formula requires that the contravariant argument of the Ext functor be finitely presented. Thus Proposition~\ref{P:AR-formula} can be viewed as an extension of the Auslander-Reiten formula to arbitrary modules. 
\end{remark}

\newcommand{\fpr}{\mathrm{fp}}
\renewcommand{\mod}{\mathrm{mod}}

\section{The small functor category and the colimit extension}\label{S:small-functor-cat}

Proposition~\ref{P:4-term-tensor} seems to suggest that restricting additive functors to finitely presented modules may provide additional insights. In this section, we shall take a closer look at this phenomenon.

In a slight change of notation, the full subcategory of right modules over $\Lambda$ determined by \texttt{finitely presented} modules will be denoted by $\mod(\Lambda)$. The category of all additive functors $F : \mod(\Lambda)\to \ab$ together with natural transformations between them will be denoted by $(\mod(\Lambda),\ab)$. This is an abelian category, a sequence of natural transformations being exact if and only if it is exact at each component. Gruson and Jensen establish in~\cite{GJ} that  $(\mod(\Lambda),\ab)$ is a Grothendieck category; in particular, it has enough injectives. They also show that the injectives are precisely the functors of the form $\blank\otimes M$, where $M$ runs through pure injective left modules. This yields

\begin{proposition}\label{P:f-colim-coprod}
 Any additive functor $F:\mod(\Lambda)\to \ab$ commutes with filtered colimits and coproducts.
\end{proposition}

\begin{proof}
Taking an injective copresentation
\begin{equation}\label{Eq:inj-copres}
 0 \lrt F \lrt \blank \otimes M \overset{\alpha}\lrt \blank\otimes N 
\end{equation}
of $F$, we have the desired result because the colimit is an exact functor and the tensor product functor commutes with filtered colimits and coproducts.
\end{proof}

This leads to  an important consequence: any additive functor 
$F:\mod(\Lambda)\to \ab$  can be extended to a unique functor $\overset{\to}{F}:\Mod(\Lambda)\to \ab$ on the entire module category, which commutes with filtered colimits.  This can be done as follows. First, assume that $F = \blank \otimes M$ for some left module $M$. If $A \in \Mod(\Lambda)$ is an arbitrary right module, set $\overset{\to}{F}(A) := A \otimes M$. If $f : A \to B$ is a morphism in $\Mod(\Lambda)$, set $\overset{\to}{F}(f) := f \otimes M$. Clearly,
$\overset{\to}F$ is a functor. The uniqueness of~$\overset{\to}F$ follows from the two well-known facts: any module can be represented as a filtered colimit of finitely presented modules and the tensor product functor commutes with filtered colimits. This construction shows that if $F$ is the tensor product functor, then
$\overset{\to}{F} (A) \simeq \varinjlim F(A_{i})$, where $A \simeq \varinjlim A_{i}$ is a representation of the right $\Lambda$-module $A$ as a filtered colimit of finitely presented modules. Before passing to the general case, we also want to define extensions of natural transformations between tensor product functors with pure injective arguments. Let $\alpha : \blank \otimes M \lrt \blank \otimes N$ be such a transformation with $M$ and $N$ pure injective. By~\cite[Proposition 1.3]{GJ},
\[
\Ext^{n}(\blank \otimes M, \blank \otimes N) \simeq \mathrm{Pext}^{n}_{\Lambda}(M,N)
\] 
where $\mathrm{Pext}^{n}_{\Lambda}(M,N)$ is the $n$th homology group of the complex $\Hom_{\Lambda}(M, Q(N))$, with $Q(N)$ being a pure injective resolution of $N$. Specializing to the case $n=0$ and using the fact that $N$ is pure injective by assumption, we have
\[
\Hom (\blank \otimes M, \blank \otimes N) \simeq \Hom_{\Lambda}(M,N)
\]
In view of this, $\alpha = \blank \otimes f$ for some $f : M \to N$. We can now define the desired extension 
\[
\overset{\to}\alpha : \blank \overset{\to}\otimes M \lrt \blank \overset{\to}\otimes N
\]
of $\alpha$ by setting $\overset{\to}\alpha := \blank \otimes f$. Since the tensor product is a bifunctor, $\overset{\to}\alpha$ is clearly a natural transformation. 
We remark that, once the injective copresentation~\eqref{Eq:inj-copres} has been chosen, the extension of $\alpha$ becomes unique up to isomorphism. If $A \simeq \varinjlim A_{i}$ is an arbitrary right $\Lambda$-module represented as a filtered colimit of finitely presented modules, then for each $A_{i}$ we have a homomorphism $\alpha_{A_{i}} : A_{i} \otimes M \to A_{i} \otimes N$ and, passing to the colimits, we have $\overset{\to}\alpha_{A} = \varinjlim \alpha_{A_{i}}$.

Now let $F : \mod(\Lambda)\to \ab$ be an arbitrary additive functor, with injective copresentation~\eqref{Eq:inj-copres}. Define $\overset{\to}F$ by the exact sequence
\[
0 \lrt \overset{\to}{F} \lrt \blank \overset{\to}\otimes M \overset{\overset{\to}\alpha}\lrt \blank \overset{\to}\otimes N
\]
Notice first that the foregoing argument for the tensor product shows that the choice of the injective copresentation of $F$ determines $\overset{\to}{F}$ uniquely up to isomorphism. Thus, to show that $\overset{\to}{F}$ is unique, it remains to show that any other choice of the injective copresentation of~$F$ yields a functor isomorphic to the same $\overset{\to}{F}$. To prove this, we recall the basic fact that any two injective resolutions are homotopy equivalent. Any homotopy, being a natural transformation between tensor products with pure injective arguments arises from a homomorphism between those arguments. This implies that the original homotopy extends to a homotopy (just evaluate on $\Lambda$). Since 
$\overset{\to}{F}$ is defined as the zeroth homology group, we have the desired uniqueness.

\begin{definition}
 The functor $\overset{\to}{F}$ will be called the colimit extension of $F$.
\end{definition}

The following known result is now obvious.

\begin{theorem}\label{T:colim-extension}
The colimit extension $\overset{\to}{F}$ of $F$ is the unique, up to isomorphism, functor on $\Mod(\Lambda)$ that commutes with filtered colimits and agrees with $F$ when restricted to $\mod(\Lambda)$. Moreover, colimit extension is an equivalence between the small functor category $(\mod(\Lambda),\ab)$ and the category of all functors $F:\Mod(\Lambda)\to \ab$ which commute with filtered colimits, the quasi-inverse provided by the restriction to finitely presented modules. \qed
\end{theorem}

Thus, we have another description of the small functor category: any additive functor $F : \mod(\Lambda) \to \ab$ can be identified with its extension $\overset{\to}{F}:\Mod(\Lambda)\to \ab$, and any functor $F:\Mod(\Lambda)\to \ab$ that commutes with filtered colimits can be identified, via its restriction, with the corresponding functor in $(\mod(\Lambda),\ab)$. Because both points of view have advantages, we will freely move between the two. 

\begin{example}
 Any functor on the large module category commuting with filtered colimits is the colimit extension of its restriction. In particular,  the colimit extension of the tensor product $\blank \otimes B$ is the same tensor product applied to all right modules.
\end{example}

%Because domain extension is an equivalence of categories and extension of tensor products yields tensor products on the entire category $\Mod(\Lambda)$, we have

\begin{corollary}\label{C:comm-coprod}
Any functor $F:\Mod(\Lambda)\to \ab$ that commutes with filtered colimits also commutes with coproducts.\footnote{This could also be seen directly because any direct sum can also be represented as a filtered direct sum.}
\end{corollary}

\begin{proof}If $F:\Mod(\Lambda)\to \ab$ commutes with filtered colimits, then it is the extension of its restriction.  Since its restriction $F\in (\mod(\Lambda),\ab)$ has injective copresentation 
\[
0\to F\to \blank\otimes M\to \blank\otimes N
\]
and since extending this exact sequence results in an exact sequence of functors on $\Mod(\Lambda)$ $$0\to F\to \blank\otimes M\to \blank \otimes N$$ the functor $F$ can be viewed as a kernel of a natural transformation between tensor functors. Since the tensor product functor commutes with coproducts and coproducts are exact, the result follows.
\end{proof}

The next proposition provides a nontrivial example of colimit extension.

\begin{proposition}Let $B$ be a left module and $A=\underset{\longrightarrow}{\lim}\ A_i$ a right module expressed as a filtered colimit of finitely presented modules.   Then $$(\blank\ot B)(A) \simeq \underset{\longrightarrow}{\lim}\Ext^1(\Tr(A_i),B)$$\end{proposition}

\begin{proof}Since $\blank\ot B$ commutes with filtered colimits, we have
%This follows from \cite[Corollary~(2.9)]{AB},  and the fact that the .  More precisely, 
\begin{eqnarray*}
(\blank\ot B)(A)
& \simeq	& (\blank\ot B)(\underset{\longrightarrow}{\lim}\ A_i) \\
& \simeq	&\underset{\longrightarrow}{\lim}\ (\blank\ot B)(A_i)\\
& \simeq	&\underset{\longrightarrow}{\lim}\ (A_{i}\ot B)\\
& \simeq	& \underset{\longrightarrow}{\lim}\Ext^1(\Tr(A_i),B)
\end{eqnarray*} 
where the last isomorphism follows from the fact that
$A_{i} \ot B \simeq \Ext^{1}(\Tr(A_{i}),B)$ for any finitely presented module $A_i$,  as discussed in Proposition~\ref{P:4-term-tensor}.
\end{proof}

\begin{example}
In ~\cite{H}, I.~Herzog studies a torsion theory on the functor category 
$(\mod(\Lambda),\ab)$. 
%There he defines for any functor $F\in(\mod(\Lambda),\ab)$ its torsion subfunctor $t(F)$.  
In particular, his torsion subfunctor $t(\blank\otimes B)$ of $\blank\otimes B$ is described by 
\[
t(\blank\otimes B) \simeq \Ext^1(\Tr(\blank),B)
\]
As a result, $\blank\ot B \simeq t(\blank\otimes B)$ when $\blank\ot B$ is restricted to the category $\mod(\Lambda)$. Thus, $\blank\ot B$ is the colimit extension of Herzog's torsion functor $t(\blank\otimes B)$.
\end{example}

\end{document}